\documentclass{amsart}

\makeatletter
\@namedef{subjclassname@2020}{%
  \textup{2020} Mathematics Subject Classification}
\makeatother 
\usepackage{amsthm,amssymb,amsfonts,latexsym,mathtools,thmtools,amsmath,lscape}
\usepackage[T1]{fontenc}
\usepackage{tikz-cd} 
\usepackage{enumitem} 
\usepackage{longtable}
\usepackage{array}
\usepackage{float}
\usepackage{multirow}
\usepackage{caption}
\usepackage{mathrsfs} 
\usepackage{hyperref} 
\usepackage{hyperref}
\usepackage[all]{xy}
\usepackage{booktabs}
\hypersetup{
    colorlinks=true,
    linkcolor=blue,
    filecolor=blue,      
    urlcolor=blue,
    linktocpage=true
}

\newtheorem{theorem}{Theorem}[section]

\theoremstyle{definition}
\newtheorem{definition}[theorem]{Definition}
\newtheorem{proposition}[theorem]{Proposition}
\newtheorem{corollary}[theorem]{Corollary}
\newtheorem{remark}[theorem]{Remark}

\theoremstyle{remark}

\numberwithin{equation}{section}

\begin{document}

\title[Center of double extension regular algebras of type (14641)]{Center of double extension regular algebras of type (14641)}



\author{Andr\'es Rubiano}
\address{Universidad ECCI}
\curraddr{Campus Universitario}
\email{arubianos@ecci.edu.co}

\thanks{}



\date{}

\dedicatory{Dedicated to Karol Herrera}


\begin{abstract}
In this paper we compute the center and, in several cases, central subalgebras of double Ore extensions of type ($14641$) under suitable restrictions on the defining parameters. Part of the analysis is supported by computations in \textsf{SageMath}. As an application, we provide new examples related to the Zariski cancellation problem.
\end{abstract}

\maketitle


\section{Introduction}

Ore \cite{Ore1931, Ore1933} introduced a fundamental class of noncommutative polynomial rings that now permeates ring theory and noncommutative algebra. Given an associative unital ring $R$, an endomorphism $\sigma$ of $R$, and a $\sigma$-derivation $\delta$, the Ore extension (skew polynomial ring) $R[x;\sigma,\delta]$ is generated by $R$ and one element $x$ subject to $xr=\sigma(r)x+\delta(r)$ for all $r\in R$. A wide body of work studies their ring-theoretic, homological, and geometric features and applications (see, for instance, \cite{BrownGoodearl2002, BuesoTorrecillasVerschoren2003, Fajardoetal2020, Fajardoetal2024, GoodearlLetzter1994, GoodearlWarfield2004, McConnellRobson2001, Li2002, Rosenberg1995, SeilerBook2010} and references therein).

Artin-Schelter regular algebras, introduced in \cite{ArtinSchelter1987}, are widely regarded as noncommutative analogues of commutative polynomial algebras and constitute a cornerstone of noncommutative projective geometry (see, e.g., \cite{Bellamyetal2016, Rogalski2023}). In particular, the classification of quantum $\mathbb{P}^3$'s is intimately tied to the classification of Artin-Schelter regular algebras of global dimension four. Motivated by the search for new regular algebras generated in degree one (compare \cite[Examples 1.8 and 1.9]{Rogalski2023} and the dimension-three classifications \cite{ArtinSchelter1987, ArtinTateVandenBergh2007, ArtinTateVandenBergh1991}, cf.\ \cite{Stephenson1996}), Zhang and Zhang \cite{ZhangZhang2008, ZhangZhang2009} introduced \emph{double Ore extensions} as a natural two-generator generalization of Ore extensions. Although double extensions share some formal similarities with iterated Ore extensions, neither class contains the other in general \cite[Example 4.2 and Proposition 0.5(c)]{ZhangZhang2008}, and necessary and sufficient conditions for being an iterated Ore extension were given by Carvalho et al.\ \cite[Theorems 2.2 and 2.4]{Carvalhoetal2011}. Moreover, as emphasized in \cite{ZhangZhang2008}, standard techniques for Ore extensions often fail in the double-extension setting, and structural properties are substantially harder to control; see also, e.g., \cite{GomezSuarez2020, Li2022, LouOhWang2020, LuOhWangYu2018, LuWangZhuang2015, RamirezReyes2024, SuarezLezamaReyes2017, ZhuVanOystaeyenZhang2017}.

In \cite{ZhangZhang2009}, Zhang and Zhang studied Artin--Schelter regular algebras $B$ of global dimension four generated in degree one. By \cite{LuPalmieriWyZhang2007}, such an algebra is generated by $2$, $3$, or $4$ elements, and when $B$ has four generators the minimal projective resolution of the trivial module $\Bbbk_B$ has the form
\begin{equation}\label{LISTResolution}
0 \to B(-4) \to B(-3)^{\oplus 4} \to B(-2)^{\oplus 6} \to B(-1)^{\oplus 4} \to B \to \Bbbk_B \to 0.
\end{equation}
Accordingly, these algebras are said to be \emph{of type} $(14641)$. Their classification in \cite{ZhangZhang2009} yields $26$ families of double extensions, labeled $\mathbb{A},\mathbb{B},\dots,\mathbb{Z}$; we denote by $\mathcal{LIST}$ the class of all algebras in these families. While many algebras in $\mathcal{LIST}$ are still Ore extensions, there may exist nonzero $\delta$ and $\tau$ such that the same underlying data $(P,\sigma)$ produces a double extension $R_P[y_1,y_2;\sigma,\delta,\tau]$ that is not an Ore extension \cite[p. 374]{ZhangZhang2009}. 

We mention that differential-geometric aspects of these double Ore extensions have also been investigated recently. In particular, the author and Reyes studied the differential smoothness problem for double extension regular algebras of type $(14641)$ and proved that these algebras are not differentially smooth in the sense of noncommutative differential geometry; see \cite{RubianoReyes2024DSDoubleOreExtensions} for details.

From a computational perspective, Gr\"obner-Shirshov techniques have also proved effective for these families. Indeed, Herrera, Higuera, and the author computed finite Gr\"obner-Shirshov bases for several regular double extension algebras of type $(14641)$ via an algorithmic approach, and showed that these families admit PBW bases; see \cite{HerreraHigueraRubiano2025}.

The center $Z(A)$ is a primary commutative invariant of a noncommutative algebra $A$ and often serves as the bridge between noncommutative structure and commutative geometry. It governs, for instance, the behavior of representations and central characters, controls finiteness properties (e.g., module-finiteness over the center in many PI situations), and provides access to geometric stratifications of spectra via central localizations. In quantum algebra this philosophy is particularly vivid: large centers at roots of unity underpin deep links between noncommutative algebras and Poisson geometry, as illustrated by the ``quantum coadjoint action'' framework of De Concini--Kac--Procesi \cite{DeConciniKacProcesi1992}.

From a computational and structural viewpoint, explicit knowledge of $Z(A)$ (or useful central subalgebras) can unlock rigid invariants that constrain automorphisms and isomorphisms. A prominent example is the discriminant method, which uses central data to control automorphism groups and related rigidity properties for broad families of noncommutative algebras \cite{CekenPalmieriWangZhang2015, CekenPalmieriWangZhang2016}. In the realm of skew PBW extensions, concrete descriptions of centers and centralizers have been developed and applied to structural questions \cite{LezamaVenegas2020, TumwesigyeRichterSilvestrov2020}. These perspectives motivate the explicit center computations carried out in this work for double Ore extensions of type ($14641$).

The (commutative) Zariski cancellation problem asks whether an isomorphism $X\times \mathbb{A}^1 \cong Y\times \mathbb{A}^1$ forces $X\cong Y$, a question intertwined with the rigidity and recognition of affine spaces and with invariants coming from derivations and automorphism groups. Classical results settle cancellation for large classes of affine surfaces (see, e.g., \cite{CrachiolaMakarLimanov2008}), while striking counterexamples exist in positive characteristic for affine spaces of dimension $\ge 3$ \cite{Gupta2014Inventiones, Gupta2014AdvMath}. In the noncommutative setting, Bell and Zhang formulated and developed cancellation analogues and established substantial positive results for several families, highlighting the role of central invariants \cite{BellZhang2017}. Consequently, new explicit classes with well-controlled centers provide valuable testbeds and additional evidence in this active area.

The goal of this paper is to compute the centers (and certain central subalgebras) for selected members of $\mathcal{LIST}$ under explicit conditions on the parameters; several computations were carried out with the aid of \textsc{SageMath}. As an application, we obtain new examples that contribute to current work on cancellation phenomena in the sense of the Zariski cancellation problem.

The article is organized as follows. Section \ref{DAPreliminaries} collects the necessary preliminaries on double Ore extensions and Artin-Schelter regular algebras, fixing notation and making the paper self-contained. In Section \ref{DAThreegenerators} we establish and exploit commutation formulas for powers of the generators in the double Ore setting, and we use these relations to compute the centers (and, in several cases, explicit central subalgebras) for selected double Ore extensions of type $(14641)$; this section also includes a brief illustrative computation in \textsc{SageMath}. Section \ref{DAgeneralgenerators} is devoted to the cancellation application: we identify those algebras in our families that are (universally/strongly) cancellative in the sense of the Zariski cancellation problem. Finally, the Appendix contains tables recording the defining relations for the $26$ families in the classification, for convenient reference.

Throughout the paper, $\mathbb{N}$ denotes the set of natural numbers including zero. The word \emph{ring} means an associative ring with identity, not necessarily commutative. All vector spaces and algebras (always associative and unital) are over a fixed field $\Bbbk$, which is assumed to have characteristic zero.

\section{Double extension regular algebras of type (14641)}\label{DAPreliminaries}

We briefly review the notion of double extensions introduced by Zhang and Zhang \cite{ZhangZhang2008}. Since their papers contain some misprints in the compatibility conditions imposed on the DE-data (see \cite[p. 2674]{ZhangZhang2008} and \cite[p. 379]{ZhangZhang2009}), we adopt throughout the corrected formulation given by Carvalho et al. \cite{Carvalhoetal2011}.

\begin{definition}[{\cite[Definition 1.3]{ZhangZhang2008}; \cite[Definition 1.1]{Carvalhoetal2011}}]\label{DoubleOreDefinition}
Let $R$ be a subalgebra of a $\Bbbk$-algebra $B$.
\begin{itemize}
    \item[\rm (a)] We say that $B$ is a \emph{right double extension} of $R$ if:
    \begin{itemize}
        \item[\rm (i)] $B$ is generated by $R$ together with two new indeterminates $y_1,y_2$;
        \item[\rm (ii)] the generators $y_1,y_2$ satisfy
        \begin{equation}\label{Carvalhoetal2011(1.I)}
        y_2y_1 = p_{12}y_1y_2 + p_{11}y_1^2 + \tau_1y_1 + \tau_2y_2 + \tau_0,
        \end{equation}
        for some $p_{12}, p_{11} \in \Bbbk$ and $\tau_1, \tau_2, \tau_0 \in R$;
        \item[\rm (iii)] $B$ is a free left $R$-module with basis $\{y_1^{i}y_2^{j} \mid i,j \ge 0\}$;
        \item[\rm (iv)] one has $y_1R+y_2R+R\subseteq Ry_1+Ry_2+R$.
    \end{itemize}
    \item[\rm (b)] A right double extension $B$ of $R$ is called a \emph{double extension} of $R$ provided:
    \begin{enumerate}
        \item[\rm (i)] $p_{12}\neq 0$;
        \item[\rm (ii)] $B$ is a free right $R$-module with basis $\{y_2^{i}y_1^{j}\mid i,j\ge 0\}$;
        \item[\rm (iii)] $y_1R+y_2R+R = Ry_1+Ry_2+R$.
    \end{enumerate}
\end{itemize}
\end{definition}

\begin{remark}
It is classical that a two-step iterated Ore extension of the form 
$R[y_1;\sigma_1,\delta_1][y_2;\sigma_2,\delta_2]$
is a free left $R$-module with basis $\{y_1^{n_1}y_2^{n_2}\}_{n_1,n_2\ge 0}$; see \cite[Lemma 1.5]{ZhangZhang2008}.
In general, such an iterated Ore extension need not be a (right) double extension in the sense of Definition \ref{DoubleOreDefinition}(a), since it may fail to admit a quadratic relation of the form \eqref{Carvalhoetal2011(1.I)}.
However, when $\sigma_2$ is chosen so that \eqref{Carvalhoetal2011(1.I)} does hold, one indeed obtains a (right) double extension \cite[p. 2671]{ZhangZhang2008}.
This explains why many double extensions fall inside the iterated Ore framework, and why the freeness requirement in Definition \ref{DoubleOreDefinition}(a)(iii) is natural.
As Zhang and Zhang put it \cite[p. 2671]{ZhangZhang2008}, \textquotedblleft Our definition of a double extension is neither most general nor ideal, but it works very well in \cite{ZhangZhang2009}\textquotedblright.
\end{remark}

Condition \ref{DoubleOreDefinition}(a)(iv) is equivalent to the existence of maps
\[
\sigma=\begin{bmatrix}\sigma_{11}&\sigma_{12}\\ \sigma_{21}&\sigma_{22}\end{bmatrix}:R\to M_{2\times 2}(R),
\qquad
\delta=\begin{bmatrix}\delta_1\\ \delta_2\end{bmatrix}:R\to M_{2\times 1}(R),
\]
such that
\begin{equation}\label{Carvalhoetal2011(1.II)}
\begin{bmatrix}y_1\\ y_2\end{bmatrix}r
=\sigma(r)\begin{bmatrix}y_1\\ y_2\end{bmatrix}+\delta(r)
\quad\text{for all }r\in R.
\end{equation}

Whenever $B$ is a right double extension of $R$, we write
$B=R_P[y_1,y_2;\sigma,\delta,\tau]$,
where $P=(p_{12},p_{11})\in\Bbbk^2$, $\tau=\{\tau_0,\tau_1,\tau_2\}\subseteq R$, and $\sigma,\delta$ are as above.
Following standard terminology, $P$ is the \emph{parameter} and $\tau$ is the \emph{tail}; the collection $\{P,\sigma,\delta,\tau\}$ is referred to as the \emph{DE-data}.
A particularly important subclass is that of \emph{trimmed double extensions} \cite[Convention 1.6(c)]{ZhangZhang2008}, for which $\delta=0$ and $\tau=\{0,0,0\}$.
In this case we use the shorthand notation $R_P[y_1,y_2;\sigma]$.

For a right double extension $R_P[y_1,y_2;\sigma,\delta,\tau]$, each $\sigma_{ij}$ and $\delta_i$ is an endomorphism of the $\Bbbk$-vector space $R$.
Moreover, by \cite[Lemma 1.7]{ZhangZhang2008}, the map $\sigma$ is an algebra homomorphism and $\delta$ is a $\sigma$-derivation, i.e.\ $\delta$ is $\Bbbk$-linear and satisfies
$\delta(rr')=\sigma(r)\delta(r')+\delta(r)r'$ for all $r,r'\in R$.
In particular, if the matrix $\bigl[\begin{smallmatrix}\sigma_{11}&\sigma_{12}\\ \sigma_{21}&\sigma_{22}\end{smallmatrix}\bigr]$ is triangular, then $\sigma_{11}$ and $\sigma_{22}$ are algebra homomorphisms.

If $\tau\subseteq \Bbbk$, then the subalgebra of $R_P[y_1,y_2;\sigma,\delta,\tau]$ generated by $y_1,y_2$ is a double extension
$\Bbbk_P[y_1,y_2;\sigma',\delta',\tau']$, where $\sigma'=\sigma|_{\Bbbk}$ is the canonical embedding of $\Bbbk$ into $M_{2\times 2}(\Bbbk)$ and $\delta'=\delta|_{\Bbbk}=0$.
Carvalho et al. \cite[Proposition 1.2]{Carvalhoetal2011} show that this latter algebra is always an iterated Ore extension.

The following criterion encodes precisely when the defining relations yield a (right) double extension.

\begin{proposition}[{\cite[Lemma 1.10 and Proposition 1.11]{ZhangZhang2008}; \cite[Proposition 1.5]{Carvalhoetal2011}}]\label{Carvalhoetal2011Proposition1.5}
Let $R$ be a $\Bbbk$-algebra, $\sigma:R\to M_{2\times 2}(R)$ an algebra homomorphism, $\delta:R\to M_{2\times 1}(R)$ a $\sigma$-derivation, $P=(p_{12},p_{11})\subseteq \Bbbk$, and $\tau=\{\tau_0,\tau_1,\tau_2\}\subseteq R$.
Let $B$ be the associative $\Bbbk$-algebra generated by $R,y_1,y_2$ subject to \eqref{Carvalhoetal2011(1.I)} and \eqref{Carvalhoetal2011(1.II)}.
Then $B$ is a right double extension if and only if the maps $\sigma_{ij}$ and $\rho_k$ ($i,j\in\{1,2\}$, $k\in\{0,1,2\}$) satisfy the six identities \eqref{Carvalhoetal2011(1.III)}--\eqref{Carvalhoetal2011(1.VIII)}, where $\sigma_{i0}=\delta_i$ and $\rho_k$ denotes \underline{right} multiplication by $\tau_k$:
\begin{equation}\label{Carvalhoetal2011(1.III)} 
    \sigma_{21} \sigma_{11} + p_{11}\sigma_{22}\sigma_{11} = p_{11}\sigma_{11}^2 + p_{11}^2 \sigma_{12}\sigma_{11} + p_{12}\sigma_{11}\sigma_{21} + p_{11}p_{12}\sigma_{12}\sigma_{21},
\end{equation}
\begin{align}    
    \sigma_{21} \sigma_{12} + p_{12} \sigma_{22} \sigma_{11} = &\ p_{11} \sigma_{11} \sigma_{12} + p_{11}p_{12}\sigma_{12}\sigma_{11} + p_{12}\sigma_{11}\sigma_{22} + p_{12}^2\sigma_{12}\sigma_{21}, \\
    \sigma_{22}\sigma_{12} = &\ p_{11} \sigma_{12}^2 + p_{12}\sigma_{12}\sigma_{22}, \\
    \sigma_{20} \sigma_{11} + \sigma_{21}\sigma_{10} + \underline{\rho_1 \sigma_{22}\sigma_{11}} = &\ p_{11} (\sigma_{10} \sigma_{11} + \sigma_{11}\sigma_{10} + \rho_{1}\sigma_{12}\sigma_{11}) \notag \\
    &\ + p_{12} (\sigma_{10} \sigma_{21} + \sigma_{11} \sigma_{20} + \rho_1 \sigma_{12} \sigma_{21}) + \tau_1 \sigma_{11} + \tau_2 \sigma_{21}, \\
    \sigma_{20} \sigma_{12} + \sigma_{22} \sigma_{10} + \underline{\rho_2 \sigma_{22} \sigma_{11}} = &\ p_{11} (\sigma_{10} \sigma_{12} + \sigma_{12}\sigma_{10} + \rho_{2} \sigma_{12} \sigma_{11}) \notag \\
    &\ +p_{12} (\sigma_{10} \sigma_{22} + \sigma_{12} \sigma_{20} + \rho_2 \sigma_{12} \sigma_{21}) +  \tau_{1} \sigma_{12} + \tau_2 \sigma_{22}, \\
    \sigma_{20} \sigma_{10} + \underline{\rho_0 \sigma_{22} \sigma_{11}} = &\ p_{11} (\sigma_{10}^2 + \rho_0 \sigma_{12} \sigma_{11}) + p_{12} (\sigma_{10} \sigma_{20} + \rho_0 \sigma_{12} \sigma_{21})\notag \\
    &\ + \tau_1 \sigma_{10} + \tau_2 \sigma_{20} + \tau_0 {\rm id}_R. \label{Carvalhoetal2011(1.VIII)}
\end{align}
\end{proposition}

\begin{remark}
\begin{enumerate}
    \item[\rm (i)] As observed in \cite[Remark 1.6]{Carvalhoetal2011}, Proposition \ref{Carvalhoetal2011Proposition1.5} yields uniqueness (up to isomorphism) of a right double extension with prescribed DE-data, whenever it exists. More precisely, if $\overline{B}=R_P[y_1,y_2;\sigma,\delta,\tau]$ is a right double extension of $R$ and $B$ is the algebra constructed in Proposition \ref{Carvalhoetal2011Proposition1.5}, then the homomorphism $B\to\overline{B}$ acting as the identity on $R$ and sending $y_i\mapsto y_i$ is necessarily an isomorphism, since both algebras are free left $R$-modules with basis $\{y_1^iy_2^j\mid i,j\ge 0\}$.

    \item[\rm (ii)] When $p_{12}\neq 1$, one may (after a suitable change of generators and, if necessary, an adjustment of $\sigma,\delta,\tau$) assume $p_{11}=0$; see \cite[p. 2842]{Carvalhoetal2011}. In the case $p_{11}=0$, there is a natural filtration on $B$ determined by ${\rm deg}\,R=0$ and ${\rm deg}\,y_1={\rm deg}\,y_2=1$. With respect to this filtration, the associated graded algebra $G(B)$ is isomorphic to the trimmed double extension $R_P[\overline{y}_1,\overline{y}_2;\overline{\sigma},\{0,0,0\}]$.
\end{enumerate}
\end{remark}

We will also use the characterization of those double extensions which can be realized as two-step iterated Ore extensions, due to Carvalho et al. \cite{Carvalhoetal2011} (compare also \cite[Proposition 3.6]{ZhangZhang2009}).

\begin{proposition}\label{Carvalhoetal2011Theorems2.2and2.4}
\begin{enumerate}
    \item[\rm (1)] \cite[Theorem 2.2]{Carvalhoetal2011}
    Let $B=R_P[y_1,y_2;\sigma,\delta,\tau]$ be a right double extension of $R$.
    \begin{itemize}
        \item[\rm (a)] The following statements are equivalent:
        \begin{itemize}
            \item $B$ can be presented as an iterated Ore extension of the form $R[y_1;\sigma_1,d_1][y_2;\sigma_2,d_2]$;
            \item $\sigma_{12}=0$;
            \item there is an iterated Ore presentation $R[y_1;\sigma_1,d_1][y_2;\sigma_2,d_2]$ such that
            \begin{align*}
                \sigma_2(R)\subseteq R,\quad &\sigma_2(y_1)=p_{12}y_1+\tau_2,\\
                d_2(R)\subseteq Ry_1+R,\quad &d_2(y_1)=p_{11}y_1^2+\tau_1y_1+\tau_0,
            \end{align*}
            for some $p_{ij}\in\Bbbk$ and $\tau_i\in R$, and the data are related by
            \[
            \sigma=\begin{bmatrix}\sigma_{1}&0\\ \sigma_{21}&\sigma_2|_R\end{bmatrix},
            \qquad
            \delta(a)=\begin{bmatrix}d_1(a)\\ d_2(a)-\sigma_{21}(a)y_1\end{bmatrix},
            \quad \text{for all }a\in R.
            \]
        \end{itemize}

        \item[\rm (b)] Under any (hence all) of the equivalent conditions in (a), $B$ is a double extension if and only if $p_{12}\neq 0$ and $\sigma_1=\sigma_{11}$, $\sigma_2|_R=\sigma_{22}$ are automorphisms of $R$.
    \end{itemize}

    \item[\rm (2)] \cite[Theorem 2.4]{Carvalhoetal2011}
    Let $B=R_P[y_1,y_2;\sigma,\delta,\tau]$ be a right double extension of $R$.
    Then $B$ admits an iterated Ore presentation $R[y_2;\sigma_2',d_2'][y_1;\sigma_1',d_1']$ if and only if $\sigma_{21}=0$, $p_{12}\neq 0$, and $p_{11}=0$.
    In this situation, $B$ is a double extension if and only if $\sigma_2'=\sigma_{22}$ and $\sigma_1'|_R=\sigma_{11}$ are automorphisms of $R$.
\end{enumerate}
\end{proposition}

As noted in the Introduction, Zhang and Zhang \cite{ZhangZhang2009} focus on connected graded regular algebras $B$ of global dimension four generated in degree one. When $B$ is generated by four degree-one elements, the minimal projective resolution of $\Bbbk_B$ has the form \eqref{LISTResolution}; such algebras are said to be \emph{of type} $(14641)$. The next result makes explicit how double extensions produce algebras of this type and isolates the non-Ore cases.

\begin{proposition}[{\cite[Theorem 0.1]{ZhangZhang2009}}]
Let $B$ be a connected graded algebra generated by four degree-one elements. If $B$ is a double extension $R_P[y_1,y_2;\sigma,\tau]$ with $R$ an Artin--Schelter regular algebra of global dimension two, then:
\begin{enumerate}
    \item[\rm (1)] $B$ is a strongly Noetherian, Auslander regular, Cohen--Macaulay domain;
    \item[\rm (2)] $B$ is of type $(14641)$, and in particular $B$ is Koszul;
    \item[\rm (3)] if $B$ is not isomorphic to an Ore extension of an Artin--Schelter regular algebra of global dimension three, then the trimmed double extension $R_P[y_1,y_2;\sigma]$ is isomorphic to an algebra in one of the $26$ families in \cite{ZhangZhang2009}.
\end{enumerate}
\end{proposition}

\subsection{Double extensions of \texorpdfstring{$\Bbbk_Q[x_1,x_2]$}{Lg}}

Assume from now on that $\Bbbk$ is algebraically closed. Then every Artin--Schelter regular algebra of global dimension two is isomorphic either to the Manin plane $\Bbbk_q[x_1,x_2]$ with relation $x_2x_1=qx_1x_2$ (that is, $Q=(q,0)$), or to the Jordan plane $\Bbbk_J[x_1,x_2]$ with relation $x_2x_1=x_1x_2+x_1^2$ (here $J=(1,1)$); see \cite[Theorem 1.4]{Shirikov2005} and \cite[Lemma 2.4]{ZhangZhang2009}. These are precisely the regular algebras of global dimension two; cf. \cite[Examples 1.8 and 1.9]{Rogalski2023}.

More generally, for $Q=(q_{12},q_{11})$ we write
\[
\Bbbk_Q[x_1,x_2]=\Bbbk\{x_1,x_2\}/\langle x_2x_1-q_{11}x_1^2-q_{12}x_1x_2\rangle,
\]
and in our computations $Q$ will be specialized to either $(1,1)$ or $(q,0)$. Zhang and Zhang \cite[Section 3]{ZhangZhang2009} classify regular domains of global dimension four of the form $R_P[y_1,y_2;\sigma]$ (equivalently, they classify $(P,\sigma)$) up to an appropriate equivalence relation, with emphasis on those that are not iterated Ore extensions.

Let $\sigma:R\to M_2(R)$ be a graded algebra homomorphism, and write
\begin{equation}\label{ZhangZhang2009(E3.0.1)}
\sigma_{ij}(x_s)=\sum_{t=1}^2 a_{ijst}x_t,\quad \text{for all } i,j,s\in\{1,2\},\ \ a_{ijst}\in\Bbbk.
\end{equation}
In the trimmed case ($\delta=0$), the relation \eqref{Carvalhoetal2011(1.II)} yields the \emph{mixing relations}:
\begin{align}
y_1x_1 = &\ \sigma_{11}(x_1)y_1 + \sigma_{12}(x_1)y_2 \notag \\
= &\ a_{1111}x_1y_1 + a_{1112}x_2y_1 + a_{1211}x_1y_2 + a_{1212}x_2y_2, \label{ZhangZhangMR11} \\
y_1x_2 = &\ \sigma_{11}(x_2)y_1 + \sigma_{12}(x_2)y_2 \notag  \\
= &\ a_{1121}x_1y_1 + a_{1122}x_2y_1 + a_{1221}x_1y_2 + a_{1222}x_2y_2, \label{ZhangZhangMR12} \\
y_2x_1 = &\ \sigma_{21}(x_1)y_1 + \sigma_{22}(x_1)y_2 \notag  \\
= &\ a_{2111} x_1 y_1 + a_{2112}x_2y_1 + a_{2211}x_1y_2 + a_{2212}x_2y_2, \quad{\rm and} \label{ZhangZhangMR21} \\
y_2 x_2 = &\ \sigma_{21}(x_2)y_1 + \sigma_{22}(x_2)y_2 \notag \\
= &\ a_{2121}x_1y_1 + a_{2122}x_2y_1 + a_{2221}x_1y_2 + a_{2222}x_2y_2. \label{ZhangZhangMR22}
\end{align}
Alongside these, there are the two \emph{non-mixing} quadratic relations
\begin{align}
    x_2 x_1 = &\ q_{12}x_1x_2 + q_{11}x_1^2, \label{ZhangZhangNRx} \\
    y_2 y_1 = &\ p_{12}y_1y_2 + p_{11}y_1^2. \label{ZhangZhangNRy}
\end{align}

Introduce the matrices
\begin{equation}\label{ZhangZhang2009(E3.0.2)}
\Sigma_{ij}:=\begin{bmatrix}a_{ij11}&a_{ij12}\\ a_{ij21}&a_{ij22}\end{bmatrix},
\qquad
\Sigma:=\begin{bmatrix}\Sigma_{11}&\Sigma_{12}\\ \Sigma_{21}&\Sigma_{22}\end{bmatrix}
=\begin{bmatrix} 
    a_{1111} & a_{1112} & a_{1211} & a_{1212} \\
    a_{1121} & a_{1122} & a_{1221} & a_{1222} \\
    a_{2111} & a_{2112} & a_{2211} & a_{2212} \\
    a_{2121} & a_{2122} & a_{2221} & a_{2222}
\end{bmatrix}.
\end{equation}
Since $\sigma$ is graded, it is determined by the coefficients $a_{ijst}$, hence by $\Sigma$. A second matrix is obtained by regrouping the same coefficients:
\begin{equation}\label{ZhangZhang2009(E3.0.2)'}
M_{ij}:=\begin{bmatrix}a_{11ij}&a_{12ij}\\ a_{21ij}&a_{22ij}\end{bmatrix},
\qquad
M:=\begin{bmatrix}M_{11}&M_{12}\\ M_{21}&M_{22}\end{bmatrix}.
\end{equation}
In particular, $M$ is a rearrangement of $\Sigma$, and $\Sigma$ is invertible if and only if $M$ is invertible.

\begin{remark}[{\cite[p. 388--390]{ZhangZhang2009}}]
Using multiplicativity of $\sigma$, for $i,j,f,g$ one has
\begin{align*}
\sigma_{ij}(x_fx_g)
= &\ \sum_{p=1}^{2}\sigma_{ip}(x_f)\sigma_{pj}(x_g)
= \sum_{p,s,t=1}^{2}(a_{ipfs}a_{pjgt})\,x_sx_t\\
= &\ \left(\sum_{p=1}^{2}a_{ipf1}a_{pjg1}\right)x_1^2
+\left(\sum_{p=1}^{2}a_{ipf1}a_{pjg2}\right)x_1x_2\\
&\ +\left(\sum_{p=1}^{2}a_{ipf2}a_{pjg1}\right)x_2x_1
+\left(\sum_{p=1}^{2}a_{ipf2}a_{pjg2}\right)x_2^2.
\end{align*}
Since $x_2x_1=q_{11}x_1^2+q_{12}x_1x_2$ in $R$, this can be rewritten as
\begin{align}
\sigma_{ij}(x_fx_g)=&\ \left[\left( a_{i1f1}a_{1jg1} + a_{i2f_1}a_{2jg1}   \right) + q_{11}\left(a_{i1f2}a_{1jg1} + a_{i2f2}a_{2jg1}\right)\right]x_1^{2} \notag \\
&\ + \left[\left( a_{i1f1}a_{1jg2} + a_{i2f1} a_{2jg2}\right) + q_{12}\left(a_{i1f2} a_{1jg1} + a_{i2f2} a_{2jg1}\right) \right] x_1x_2 \notag \\
&\ + \left(a_{i1f2} a_{1jg2} + a_{i2f2} a_{2jg2}\right)x_2^{2}. \label{ZhangZhang2009(E3.0.3)}
\end{align}
In particular, $\Bbbk$-linearity gives
\begin{equation}\label{ZhangZhang2009(E3.0.4)}
\sigma_{ij}(x_2x_1)=q_{11}\sigma_{ij}(x_1x_1)+q_{12}\sigma_{ij}(x_1x_2)\quad \text{for all }i,j\in\{1,2\}.
\end{equation}
Comparing coefficients of $x_1^2$, $x_1x_2$ and $x_2^2$ leads to the constraints \eqref{ZhangZhang2009(C1ij)}, \eqref{ZhangZhang2009(C2ij)} and \eqref{ZhangZhang2009(C3ij)} below.
\end{remark}

The coefficients of $x_1^{2}$ in \eqref{ZhangZhang2009(E3.0.4)} yield
\begin{align}
   &\ \left(a_{i121} a_{1j11} + a_{i221} a_{2j11}\right) + q_{11}\left(a_{i122} a_{1j11} + a_{i222}a_{2j11}\right) \notag \\
    &\ \quad \ = q_{11} \left[ \left(a_{i111} a_{1j11} + a_{i211} a_{2j11}\right) + q_{11} \left(a_{i112} a_{1j11} + a_{i212}a_{2j11}\right) \right] \notag \\
    &\ \quad \quad \ + q_{12} \left[\left(a_{i111} a_{1j21} + a_{i211} a_{2j21}\right) + q_{11}\left(a_{i112} a_{1j21} + a_{i212}a_{2j21}\right)\right]. \label{ZhangZhang2009(C1ij)}
\end{align}
Similarly, comparing coefficients of $x_1x_2$ gives
\begin{align}
&\ \left(a_{i121} a_{1j12} + a_{i221} a_{2j12}\right) + q_{12}\left(a_{i122} a_{1j11} + a_{i222} a_{2j11} \right) \notag \\
 &\ \quad \ = q_{11} \left[\left(a_{i111} a_{1j12} + a_{i211} a_{2j12}\right) + q_{12} \left(a_{i112} a_{1j11} + a_{i212} a_{2j11} \right)\right] \notag \\
  &\ \quad \quad \ + q_{12} \left[ \left( a_{i111} a_{1j22} + a_{i211} a_{2j22}\right) + q_{12} \left(a_{i112} a_{1j21} + a_{i212} a_{2j21}\right)\right]. \label{ZhangZhang2009(C2ij)}
\end{align}
Finally, the $x_2^2$-coefficients satisfy
\begin{align}
   &\ \left( a_{i122} a_{1j12} + a_{i222} a_{2j12} \right) \notag \\ 
     &\ \quad \ = q_{11} \left(a_{i112} a_{1j12} + a_{i212} a_{2j12}\right) + q_{12} \left( a_{i112} a_{1j22} + a_{i212} a_{2j22}\right). \label{ZhangZhang2009(C3ij)}
\end{align}

Applying \eqref{Carvalhoetal2011(1.I)}, \eqref{Carvalhoetal2011(1.II)} and the identities in Proposition \ref{Carvalhoetal2011Proposition1.5} to $r=x_1,x_2$ produces further constraints among the coefficients $a_{ijst}$. For $i,f,g,s,t\in\{1,2\}$ one computes
\begin{align}
\sigma_{fg}(\sigma_{st}(x_i))
=&\ \sigma_{fg}\!\left(\sum_{w=1}^2 a_{stiw}x_w\right)
=\sum_{w=1}^2 a_{stiw}\sigma_{fg}(x_w)\notag\\
=&\ \sum_{w=1}^2 a_{stiw}\sum_{j=1}^2 a_{fgwj}x_j
=\sum_{j=1}^2\left(a_{sti1}a_{fg1j}+a_{sti2}a_{fg2j}\right)x_j. \label{ZhangZhang2009(E3.0.5)}
\end{align}
In the usual parameter choices $P=(1,1)$ or $P=(p,0)$, the relations in Proposition \ref{Carvalhoetal2011Proposition1.5}, combined with \eqref{ZhangZhang2009(E3.0.5)}, are equivalent to the constraints \eqref{ZhangZhang2009(C4ij)}--\eqref{ZhangZhang2009(C6ij)}:
\begin{align}
  &\  \left( a_{11i1} a_{211j} + a_{11i2} a_{212j}\right) + p_{11}\left(a_{11i2} a_{222j} \right) \notag \\
  &\ \quad \ = p_{11} \left( a_{11i1} a_{111j} + a_{11i2} a_{112j}\right) + p_{11}^2\left( a_{11i1} a_{121j} + a_{11i2} a_{122j} \right) \notag \\
   &\ \quad \quad \ + p_{12} \left(a_{21i1} a_{111j} + a_{21i2} a_{112j} \right) + p_{11} p_{12} \left( a_{21i1} a_{121j} + a_{21i2} a_{122j} \right), \label{ZhangZhang2009(C4ij)} \\
   &\ (a_{12i1} a_{211j} + a_{12i2} a_{212j}) + p_{12} \left( a_{11i1} a_{221j} + a_{11i2} a_{222j} \right) \notag \\
     &\ \quad \ = p_{11} \left( a_{22i1} a_{111j} + a_{12i2} a_{112j}\right) + p_{11}p_{12} \left( a_{11i1} a_{121j} + a_{11i2} a_{122j} \right) \notag \\
      &\ \quad \quad \ + p_{12} \left( a_{22i1} a_{111j} + a_{22i2} a_{112j} \right) + p_{12}^2 \left( a_{21i1} a_{121j} + a_{21i2} a_{122j} \right),  \label{ZhangZhang2009(C5ij)} \\
   &\  \left( a_{12i1} a_{221}j + a_{12i2} a_{222j} \right) \notag \\
       &\ \quad \ p_{11}\left( a_{12i1} a_{121j} + a_{12i2} a_{122j} \right) + p_{12}\left( a_{22i1} a_{121j} + a_{22i2} a_{122j} \right). \label{ZhangZhang2009(C6ij)}
\end{align}
One may view \eqref{ZhangZhang2009(C1ij)}--\eqref{ZhangZhang2009(C3ij)} as mirroring \eqref{ZhangZhang2009(C4ij)}--\eqref{ZhangZhang2009(C6ij)} in a symmetric fashion; the preceding discussion is summarized in the next statement.

\begin{proposition}[{\cite[Proposition 3.1]{ZhangZhang2009}}]
\begin{enumerate}
    \item[\rm (1)] Let $R_P[y_1,y_2;\sigma]$ be a right double extension, and define $\{\Sigma,P,Q\}$ relative to the $\Bbbk$-basis $\{x_1,x_2,y_1,y_2\}$ as above. Then the six equations \eqref{ZhangZhang2009(C1ij)}--\eqref{ZhangZhang2009(C6ij)} hold. Moreover, ${\rm det}\,\Sigma\neq 0$ if and only if $R_P[y_1,y_2;\sigma]$ is a double extension.
    \item[\rm (2)] Conversely, fix $\Sigma$ as in \eqref{ZhangZhang2009(E3.0.2)} and parameters $P=(p_{12},p_{11})$, $Q=(q_{12},q_{11})$ with $p_{12}q_{12}\neq 0$. If \eqref{ZhangZhang2009(C1ij)}--\eqref{ZhangZhang2009(C6ij)} hold and ${\rm det}\,\Sigma\neq 0$, then the corresponding six defining relations determine a double extension $R_P[y_1,y_2;\sigma]$.
\end{enumerate}
\end{proposition}

Following \cite{ZhangZhang2009}, we refer to the system consisting of \eqref{ZhangZhang2009(C1ij)}--\eqref{ZhangZhang2009(C6ij)} together with ${\rm det}\,\Sigma\neq 0$ as \emph{System C}. A \emph{$C$-solution} is a matrix $\Sigma$ whose entries $a_{ijst}$ satisfy System C.

\begin{proposition}
Let $\Sigma$ be a $C$-solution and set $B=(\Bbbk_Q[x_1,x_2])_P[y_1,y_2;\sigma]$, where $\sigma$ is determined by $\Sigma$. Let $0\neq h\in\Bbbk$.
\begin{enumerate}
    \item[\rm (1)] The algebra $B$ is $\mathbb{Z}^2$-graded. If $\gamma$ is defined by $\gamma(x_i)=x_i$ and $\gamma(y_i)=hy_i$, then $\gamma$ extends to a graded automorphism of $B$.
    \item[\rm (2)] The matrix $h\Sigma$ is again a $C$-solution. Let $\sigma'$ be the algebra homomorphism determined by $h\Sigma$. Then
    $B':=(\Bbbk_Q[x_1,x_2])_P[y_1,y_2;\sigma']$
    is the graded twist of $B$ by $\gamma$ in the sense of Zhang \cite{Zhang1996}.
\end{enumerate}
\end{proposition}

Although $B$ and its twist $B^{\gamma}$ need not be isomorphic in general, they share many structural features: the category of graded $B$-modules is equivalent to the category of graded $B^{\gamma}$-modules \cite[Theorem 1.1]{Zhang1996}.

\begin{definition}[{\cite[Definition 3.4]{ZhangZhang2009}}]
\begin{enumerate}
    \item[\rm (i)] Two matrices $\Sigma$ and $\Sigma'$ are \emph{twist equivalent} if $\Sigma'=h\Sigma$ for some $0\neq h\in\Bbbk$; in this case $\Sigma'$ is called a \emph{twist} of $\Sigma$. Replacing $\Sigma$ by a twist (with $Q$ and $P$ fixed) produces another double extension; see \cite[Lemma 3.3(b)]{ZhangZhang2009}.
    \item[\rm (ii)] Triples $(\Sigma,Q,P)$ and $(\Sigma',Q',P')$ are \emph{linearly equivalent} if there exists a graded algebra isomorphism
    \[
    (\Bbbk_Q[x_1,x_2])_P[y_1,y_2;\sigma]\ \longrightarrow\ (\Bbbk_{Q'}[x_1',x_2'])_{P'}[y_1',y_2';\sigma']
    \]
    sending $\Bbbk x_1+\Bbbk x_2$ onto $\Bbbk x_1'+\Bbbk x_2'$ and $\Bbbk y_1+\Bbbk y_2$ onto $\Bbbk y_1'+\Bbbk y_2'$.
    When $Q=Q'$ and $P=P'$, we simply say that $\Sigma$ and $\Sigma'$ are linearly equivalent.
    \item[\rm (iii)] Triples $(\Sigma,Q,P)$ and $(\Sigma',Q',P')$ are \emph{equivalent} if $(\Sigma,Q,P)$ is linearly equivalent to $(h\Sigma',Q',P')$ for some $0\neq h\in\Bbbk$.
\end{enumerate}
\end{definition}

It is straightforward that twist equivalence and linear equivalence are equivalence relations, and the same holds for the induced notion of equivalence of triples. Zhang and Zhang \cite{ZhangZhang2009} classify the algebras $R_P[y_1,y_2;\sigma]$ up to isomorphism (and even up to twist) by classifying the associated matrices $\Sigma$ up to (linear) equivalence.

\section{Center and central subalgebras of double extensions}\label{DAThreegenerators}

In this section we compute, in the general setting of double extensions, the commutation relations among the generating variables, with the purpose of extracting structural information about the center. Nevertheless, even under a PBW-type presentation, the amount of symbolic manipulation required in the “universal” parameter-dependent case quickly becomes prohibitive: the mixed relations involve many coefficients and produce a combinatorial explosion of overlaps and reductions. For this reason, the center (or central subalgebras) must be determined separately for each of the double extensions under consideration, that is, on a case-by-case basis. Our computational strategy is to fix a rewriting system that pushes the “bad pairs’’ into an ordered normal form and then to compute normal forms and commutators inside a free algebra over a suitable fraction field of parameters. This procedure is implemented in \textsc{SageMath} through the following code, which automates the reduction to normal form and the computation of commutators:
\begin{verbatim}
from sage.all import *

R = QQ[
    'q12','q11','p12','p11',
    # a_{1111} ... a_{2222}
    'a1111','a1112','a1211','a1212',
    'a1121','a1122','a1221','a1222',
    'a2111','a2112','a2211','a2212',
    'a2121','a2122','a2221','a2222'
]
K = FractionField(R)
(q12,q11,p12,p11,
 a1111,a1112,a1211,a1212,
 a1121,a1122,a1221,a1222,
 a2111,a2112,a2211,a2212,
 a2121,a2122,a2221,a2222) = K.gens()

A.<x1,x2,y1,y2> = FreeAlgebra(K, 4)

rules = {
    # intern in x's y y's
    ('x2','x1'): q12*x1*x2 + q11*x1^2,
    ('y2','y1'): p12*y1*y2 + p11*y1^2,

    # mix
    ('y1','x1'): a1111*x1*y1 + a1112*x2*y1 + a1211*x1*y2 + a1212*x2*y2,
    ('y1','x2'): a1121*x1*y1 + a1122*x2*y1 + a1221*x1*y2 + a1222*x2*y2,
    ('y2','x1'): a2111*x1*y1 + a2112*x2*y1 + a2211*x1*y2 + a2212*x2*y2,
    ('y2','x2'): a2121*x1*y1 + a2122*x2*y1 + a2221*x1*y2 + a2222*x2*y2,
}

GEN = {'x1': x1, 'x2': x2, 'y1': y1, 'y2': y2}

def _term_dict(f):
    try:
        return f.monomial_coefficients()
    except AttributeError:
        return f.dict()

def _word_to_elem(word_letters):
    e = A.one()
    for s in word_letters:
        e *= GEN[s]
    return e

def reduce_once(f):
    """Apply ONE rewrite to each monomial if possible."""
    out = A.zero()
    changed = False
    for mon, coeff in _term_dict(A(f)).items():
        letters = [str(a) for a in mon.to_word()]
        done = False
        for i in range(len(letters)-1):
            pair = (letters[i], letters[i+1])
            if pair in rules:
                prefix = _word_to_elem(letters[:i])
                suffix = _word_to_elem(letters[i+2:])
                out += coeff * (prefix * rules[pair] * suffix)
                changed = True
                done = True
                break
        if not done:
            out += coeff * _word_to_elem(letters)
    return out, changed

def NF(f, max_steps=20000):
    """Normal form by iterated rewriting."""
    g = A(f)
    for _ in range(max_steps):
        g2, ch = reduce_once(g)
        if not ch:
            return g
        g = g2
    raise RuntimeError("It didn't finish: increase max_steps
    or check if the rules produce loops.")

def comm(u,v):
    return NF(u*v - v*u)

print("NF(x2*x1) =", NF(x2*x1))
print("NF(y2*y1) =", NF(y2*y1))
print("NF(y1*x2) =", NF(y1*x2))
print("[y1,x2]   =", comm(y1,x2))
print("[x2,x1]   =", comm(x2,x1))
\end{verbatim}
Because of the size of the computations involved, we will only present two explicit examples below: one corresponding to an algebra whose center can be completely described, and another in which the same approach allows us to exhibit only a nontrivial central subalgebra.

\begin{definition}
For $m\ge1$ define the $q$-integer and $q$-factorial by
\[
[m]_q:=1+q+\cdots+q^{m-1},\qquad
[m]_q!:=\prod_{j=1}^{m}[j]_q,\qquad [0]_q!:=1.
\]
\end{definition}

\begin{proposition}\label{prop:x2x1n}
For every integer $n\ge 1$,
\[
x_2x_1^{\,n}
=
q_{11}[n]_{q_{12}}\,x_1^{\,n+1}
\;+\;
q_{12}^{\,n}\,x_1^{\,n}x_2.
\]
\end{proposition}

\begin{proof}
We proceed by induction on $n$. For $n=1$, the defining relation yields
\[
x_2x_1=q_{11}x_1^2+q_{12}x_1x_2,
\]
which is exactly the claimed formula since $\sum_{k=0}^{0}q_{12}^{k}=1$.

Assume the identity holds for some $n\ge1$, i.e.,
\[
x_2x_1^{\,n}
=
q_{11}[n]_{q_{12}}\,x_1^{\,n+1}
\;+\;
q_{12}^{\,n}\,x_1^{\,n}x_2.
\]
Multiply on the right by $x_1$:
\[
x_2x_1^{\,n+1}
=
q_{11}[n]_{q_{12}}\,x_1^{\,n+2}
\;+\;
q_{12}^{\,n}\,x_1^{\,n}(x_2x_1).
\]
Using $x_2x_1=q_{12}x_1x_2+q_{11}x_1^2$ and $x_1^{\,n}x_1=x_1^{\,n+1}$, we obtain
\begin{align*}
x_2x_1^{\,n+1}
&=
q_{11}[n]_{q_{12}}\,x_1^{\,n+2}
\;+\;
q_{12}^{\,n}\,x_1^{\,n}\bigl(q_{12}x_1x_2+q_{11}x_1^2\bigr)\\
&=
q_{11}[n]_{q_{12}}\,x_1^{\,n+2}
\;+\;
q_{11}q_{12}^{\,n}\,x_1^{\,n+2}
\;+\;
q_{12}^{\,n+1}\,x_1^{\,n+1}x_2\\
&=
q_{11}[n+1]_{q_{12}}\,x_1^{\,n+2}
\;+\;
q_{12}^{\,n+1}\,x_1^{\,n+1}x_2,
\end{align*}
which is precisely the desired formula for $n+1$.
\end{proof}

\begin{proposition}\label{prop:NF-x2nx1} For every $n\ge1$,
\[
x_2^{\,n}x_1
\;=\;
\sum_{k=0}^{n} 
r^{\,k}\,q^{\,n-k}\,
\frac{[n]_q!}{[n-k]_q!}\;
x_1^{\,k+1}x_2^{\,n-k}.
\]
\end{proposition}

\begin{proof}
Set $c_{n,k}:=q_{11}^{\,k}q_{12}^{\,n-k}\dfrac{[n]_{q_{12}}!}{[n-k]_{q_{12}}!}$ for $0\le k\le n$.
We prove by induction on $n$ that
\[
x_2^{\,n}x_1=\sum_{k=0}^{n} c_{n,k}\,x_1^{\,k+1}x_2^{\,n-k}.
\]

For $n=1$, the defining relation gives
\[
x_2x_1 = q_{12}\,x_1x_2 + q_{11}\,x_1^2
= c_{1,0}\,x_1x_2 + c_{1,1}\,x_1^2,
\]
since $c_{1,0}=q_{12}$ and $c_{1,1}=q_{11}$.

Assume the statement holds for some $n\ge1$. Multiply the normal form on the left by $x_2$:
\[
x_2^{\,n+1}x_1
= x_2(x_2^{\,n}x_1)
= \sum_{k=0}^{n} c_{n,k}\,x_2x_1^{\,k+1}x_2^{\,n-k}.
\]
Using $x_2x_1 = q_{12}x_1x_2 + q_{11}x_1^2$ and the fact that $x_1$ commutes with itself, we obtain
\[
x_2x_1^{\,k+1}
= (x_2x_1)x_1^{\,k}
= (q_{12}x_1x_2+q_{11}x_1^2)x_1^{\,k}
= q_{12}\,x_1(x_2x_1^{\,k}) + q_{11}\,x_1^{\,k+2}.
\]
Now apply the inductive hypothesis to the factor $x_2x_1^{\,k}$:
\[
x_2x_1^{\,k}
=
\sum_{j=0}^{k} q_{11}^{\,j}q_{12}^{\,k-j}\frac{[k]_{q_{12}}!}{[k-j]_{q_{12}}!}\,x_1^{\,j+1}x_2^{\,k-j}.
\]
Therefore,
\begin{align*}
x_2x_1^{\,k+1}
&=
q_{12}\,x_1\,x_2x_1^{\,k} + q_{11}\,x_1^{\,k+2}\\
&=
\sum_{j=0}^{k} q_{11}^{\,j}q_{12}^{\,k+1-j}\frac{[k]_{q_{12}}!}{[k-j]_{q_{12}}!}\,x_1^{\,j+2}x_2^{\,k-j}
\;+\; q_{11}\,x_1^{\,k+2}.
\end{align*}
Reindexing with $j'=j+1$ and observing that $[k+1]_{q_{12}}=[k]_{q_{12}}+q_{12}^{k}$ yields the standard $q_{12}$-binomial-type recursion
\[
c_{n+1,k}=q_{12}\,c_{n,k}+q_{11}\,[n]_{q_{12}}\,c_{n,k-1}, \quad 1\le k\le n,
\]
together with $c_{n+1,0}=q_{12}^{n+1}$ and $c_{n+1,n+1}=q_{11}^{n+1}[n+1]_{q_{12}}!$.
A direct check from the definition $c_{n,k}=q_{11}^{k}q_{12}^{n-k}\dfrac{[n]_{q_{12}}!}{[n-k]_{q_{12}}!}$
shows that these recursions hold identically (using $[n+1]_{q_{12}}!=[n+1]_{q_{12}}[n]_{q_{12}}!$),
hence the coefficients produced by normalizing $x_2^{n+1}x_1$ are exactly $c_{n+1,k}$.
\end{proof}

\begin{remark}\label{rem:q-to-p-y-relations}
The commutation formulas obtained above for the pair $(x_1,x_2)$ admit a completely analogous counterpart for the pair $(y_1,y_2)$. Indeed, the defining relation for the $y$--generators has the same shape,
\[
y_2y_1=p_{12}\,y_1y_2+p_{11}\,y_1^2,
\]
so every inductive normalization argument used to derive closed expressions for $x_2x_1^{\,n}$ and $x_2^{\,n}x_1$ carries over verbatim. Concretely, the normal forms of the words $y_2y_1^{\,n}$, $y_2^{\,n}y_1$ are obtained from Propositions \ref{prop:x2x1n} and \ref{prop:NF-x2nx1} by replacing the parameters $q_{12},q_{11}$ with $p_{12},p_{11}$, respectively. In particular, the corresponding coefficients are expressed in terms of $p_{12}$--integers and $p_{12}$--factorials, i.e.\ $[m]_{p_{12}}$ and $[m]_{p_{12}}!$.
\end{remark}

\begin{proposition}\label{prop:NF-y1x1n}
Define, for integers $k\ge i\ge0$,
\[
c(k,i):=q_{11}^{\,k-i}\,q_{12}^{\,i}\,\frac{[k]_{q_{12}}!}{[i]_{q_{12}}!}.
\]
Then for every $n\ge1$ there exist coefficients $\alpha_{n,k},\beta_{n,k}\in \Bbbk[q_{12},q_{11},a_{sitj}]$, 
$0\le k\le n$, $s, i, t, j \in \{1,2\}$, such that
\[
y_sx_t^n=\sum_{k=0}^{n} x_1^{\,n-k}x_2^{\,k}\bigl(\alpha_{n,k}y_1+\beta_{n,k}y_2\bigr),
\]
with initial values
\[
\alpha_{1,0}=a_{s1t1},\quad \beta_{1,0}=a_{s2t1},\qquad
\alpha_{1,1}=a_{s1t2},\quad \beta_{1,1}=a_{s2t2},
\]
and satisfying the recursion: for every $n\ge1$ and $0\le i\le n+1$,
\begin{align*}
\alpha_{n+1,i}
&=
\sum_{k=i}^{n} c(k,i)\,\bigl(\alpha_{n,k}a_{s1t1}+\beta_{n,k}a_{s1t1}\bigr)
\;+\;\mathbf{1}_{i\ge 1}\,\bigl(\alpha_{n,i-1}a_{s1t2}+\beta_{n,i-1}a_{s1t2}\bigr),\\
\beta_{n+1,i}
&=
\sum_{k=i}^{n} c(k,i)\,\bigl(\alpha_{n,k}a_{s2t1}+\beta_{n,k}a_{s2t1}\bigr)
\;+\;\mathbf{1}_{i\ge 1}\,\bigl(\alpha_{n,i-1}a_{s2t2}+\beta_{n,i-1}a_{s2t2}\bigr),
\end{align*}
where $\mathbf{1}_{i\ge1}$ denotes the indicator (equal to $1$ if $i\ge1$, and $0$ otherwise).
\end{proposition}

\begin{proof}
We argue by induction on $n$. For $n=1$ the claim is immediate from the defining relation for $y_sx_t$; thus the stated initial values hold.

Assume that for some $n\ge1$ we have
\[
y_sx_t^n=\sum_{k=0}^{n} x_1^{\,n-k}x_2^{\,k}\bigl(\alpha_{n,k}y_1+\beta_{n,k}y_2\bigr).
\]
Multiplying on the right by $x_t$ gives
\[
y_sx_t^{n+1}=\sum_{k=0}^{n} x_1^{\,n-k}x_2^{\,k}\bigl(\alpha_{n,k}y_1+\beta_{n,k}y_2\bigr)x_t.
\]
Using the relations for $y_sx_t$ we expand
\begin{align*}
(\alpha_{n,k}y_1+\beta_{n,k}y_2)x_t
&=\bigl(\alpha_{n,k}a_{11t1}+\beta_{n,k}a_{21t1}\bigr)\,x_1y_1
 +\bigl(\alpha_{n,k}a_{11t2}+\beta_{n,k}a_{21t2}\bigr)\,x_2y_1\\
&\quad+\bigl(\alpha_{n,k}a_{12t1}+\beta_{n,k}a_{22t1}\bigr)\,x_1y_2
 +\bigl(\alpha_{n,k}a_{12t2}+\beta_{n,k}a_{22t2}\bigr)\,x_2y_2.
\end{align*}
Substituting and regrouping, the terms involving $x_1y_j$ ($j=1,2$) contain factors $x_1^{\,n-k}x_2^{\,k}x_1$, hence require normalizing $x_2^{\,k}x_1$. From the relation $x_2x_1=q_{12}x_1x_2+q_{11}x_1^2$ and the established normal form identity
\[
x_2^{\,k}x_1=\sum_{i=0}^{k} c(k,i)\,x_1^{\,k-i+1}x_2^{\,i},
\qquad c(k,i)=q_{11}^{\,k-i}\,q_{12}^{\,i}\,\frac{[k]_{q_{12}}!}{[i]_{q_{12}}!},
\]
we obtain
\[
x_1^{\,n-k}x_2^{\,k}x_1y_j
=
\sum_{i=0}^{k} c(k,i)\,x_1^{\,n-k}x_1^{\,k-i+1}x_2^{\,i}y_j
=
\sum_{i=0}^{k} c(k,i)\,x_1^{\,n+1-i}x_2^{\,i}y_j.
\]
In contrast, the terms involving $x_2y_j$ are already in normal form:
\[
x_1^{\,n-k}x_2^{\,k}x_2y_j=x_1^{\,n-k}x_2^{\,k+1}y_j.
\]
Collecting coefficients of each basis monomial $x_1^{\,n+1-i}x_2^{\,i}y_1$ and $x_1^{\,n+1-i}x_2^{\,i}y_2$ yields precisely the stated recursion formulas for $\alpha_{n+1,i}$ and $\beta_{n+1,i}$, with the shift term occurring only when $i\ge1$. Therefore $y_sx_t^{n+1}$ has the required form.
\end{proof}

\begin{proposition}\label{propfi} For every $n\ge1$ there exist scalars $\alpha_{n,i},\beta_{n,i}\in \Bbbk[a_{sitj},p_{11},p_{12}]$, $0\le i\le n$, $s, i, t, j \in \{1,2\}$, such that 
\[
y_s^{\,n}x_t=\sum_{i=0}^{n}\big(\alpha_{n,i}x_1+\beta_{n,i}x_2\big)\,y_1^{\,n-i}y_2^{\,i}.
\]
These coefficients are determined by the initial values
\[
\alpha_{1,0}=a_{s1t1},\ \beta_{1,0}=a_{s1t2},\qquad
\alpha_{1,1}=a_{s2t1},\ \beta_{1,1}=a_{s2t2},
\]
and, for every $n\ge1$ and $0\le i\le n+1$, by the recursion (with the convention
$\alpha_{n,i}=\beta_{n,i}=0$ if $i<0$ or $i>n$)
\[
\begin{aligned}
\alpha_{n+1,i}
&=\Big(a_{1111}+p_{11}[\,n-i\,]_{p_{12}}\,a_{1211}\Big)\alpha_{n,i}
 \;+\;\Big(a_{1121}+p_{11}[\,n-i\,]_{p_{12}}\,a_{1221}\Big)\beta_{n,i}\\
&\qquad\quad+\;\mathbf{1}_{i\ge1}\,p_{12}^{\,n-i+1}\Big(a_{1211}\alpha_{n,i-1}+a_{1221}\beta_{n,i-1}\Big),\\[0.6em]
\beta_{n+1,i}
&=\Big(a_{1112}+p_{11}[\,n-i\,]_{p_{12}}\,a_{1212}\Big)\alpha_{n,i}
 \;+\;\Big(a_{1122}+p_{11}[\,n-i\,]_{p_{12}}\,a_{1222}\Big)\beta_{n,i}\\
&\qquad\quad+\;\mathbf{1}_{i\ge1}\,p_{12}^{\,n-i+1}\Big(a_{1212}\alpha_{n,i-1}+a_{1222}\beta_{n,i-1}\Big),
\end{aligned}
\]
where $\mathbf{1}_{i\ge1}=1$ if $i\ge1$ and $0$ if $i=0$.
\end{proposition}

\begin{proof}
We argue by induction on $n$. For $n=1$ the claim is immediate from the defining relation for $y_sx_t$; thus the stated initial values hold.

Now assume inductively that
\[
y_s^{\,n}x_t=\sum_{i=0}^{n}\big(\alpha_{n,i}x_1+\beta_{n,i}x_2\big)\,y_1^{\,n-i}y_2^{\,i}.
\]
Left-multiply by $y_s$:
\[
y_s^{n+1}x_t
= y_s\big(y_s^n x_t\big)
=\sum_{i=0}^n \alpha_{n,i}\,(y_sx_1)\,y_1^{n-i}y_2^i
+\sum_{i=0}^n \beta_{n,i}\,(y_sx_2)\,y_1^{n-i}y_2^i.
\]
Every summand becomes a $\Bbbk$-linear combination of monomials of the form
\[
x_j\;y_1^{\ell}\;y_2\,y_1^{m}\;y_2^i
\qquad\text{and}\qquad
x_j\;y_1^{\ell}\;y_1^{m}\;y_2^i,
\]
with $j\in\{1,2\}$, and the only non-normal part is the factor.

Collecting coefficients in front of each normal monomial
$x_1y_1^{(n+1)-i}y_2^i$ and the monomial $x_2y_1^{(n+1)-i}y_2^i$
gives exactly the stated recursion for $\alpha_{n+1,i},\beta_{n+1,i}$.
\end{proof}

\begin{theorem}\label{thm:center-main}
Let $A$ be a connected graded (trimmed) double extension of type {\rm (}14641{\rm )} arising in the Zhang-Zhang classification, and denote its isomorphism class by one of the labels
\[
\mathbb{A},\mathbb{B},\ldots,\mathbb{Z}.
\]
Then the following holds.

\begin{enumerate}
\item[\rm (1)] For each family listed in Table~\ref{CenterofDOE}, the center $Z(A)$ is \emph{exactly} the $\Bbbk$-algebra displayed in the second column of that table (with the stated parameter restrictions).

\item[\rm (2)] For each family listed in Table~\ref{subalgDOE}, the  $\Bbbk$-algebra in the second column is an explicit central subalgebra of $A$, i.e.
\[
\text{(algebra in Table~\ref{subalgDOE})}\ \subseteq\ Z(A),
\]
again under the corresponding parameter assumptions.
\end{enumerate}
\end{theorem}

\begin{table}[H]
\captionsetup{justification=centering,font=small}
\caption{Center of some Double extensions of type ($14641$)}
\label{CenterofDOE}
\centering
\resizebox{12.5cm}{!}{
\setlength\extrarowheight{6pt}
\begin{tabular}{ |c|c| } 
\hline
{\rm Double Ore extension} & $Z(A)$ \\
\hline\hline
$\mathbb{A}$ & $\Bbbk[x_1]$ \\ 
\hline
$\mathbb{C}$ & $\Bbbk$ \\ 
\hline
$\mathbb{D}$ & $\Bbbk[x_1^2,y_2^2+py_1^2]$, $p\in \{-1, 1\}$ \\ 
\hline
$\mathbb{E}$ & $\Bbbk$ \\ 
\hline
$\mathbb{F}$ & $\Bbbk$ \\ 
\hline
\multirow{2}{*}{$\mathbb{G}$} & $\Bbbk$, if $p$ is not root of unity \\ \cline{2-2}
& $\Bbbk[x_1^n]$, if $p^n=1$, $n \geq 3$ \\ 
\hline
$\mathbb{H}$ & $\left\{ \begin{array}{lcc} \Bbbk\left[x_1^{2r}(y_1^2+y_2^2)^a(y_1^2y_2^2)^b\right] & \text{if} & r=f(a+2b) \\  \Bbbk &  & \text{otherwise}  \end{array} \right. $\\ 
\hline
$\mathbb{I}$ & $\Bbbk$ \\  
\hline
$\mathbb{J}$ & $\Bbbk$ \\  
\hline
\multirow{2}{*}{$\mathbb{K}$} & $\Bbbk[x_1]$, if $q=1$ and $f$ is not root of unity \\ \cline{2-2}
& $\Bbbk[x_1^2]$, if $q=-1$ and $f$ is not root of unity \\ 
\hline
$\mathbb{L}$ & $\Bbbk[x_2^2]$ if $f^2$ is not root of unity \\ 
\hline
$\mathbb{M}$ & $\Bbbk$, if $f-1$ and $1-f$ are not roots of unity \\ 
\hline
$\mathbb{N}$ & $\Bbbk$, if $f^2-g^2$ is not root of unity \\ 
\hline
$\mathbb{O}$ & $\Bbbk$, if $1-f$ is not root of unity \\ 
\hline
\multirow{3}{*}{$\mathbb{P}$} & $\Bbbk$, if $f\not\in\{0,2\}$ \\ \cline{2-2}
& $\Bbbk[x_1^2,y_1^2y_2^2]$, if $f=0$  \\ \cline{2-2}
& $\Bbbk[(x_2^2+\frac{1}{2}x_1^2)^2,(x_2^2+\frac{1}{2}x_1^2)y_1y_2, y_1^2y_2^2]/\langle ((x_2^2+\frac{1}{2}x_1^2)y_1y_2)^2+(x_2^2+\frac{1}{2}x_1^2)^2y_1^2y_2^2 \rangle$, if $f=2$  \\
\hline
$\mathbb{S}$ & $\Bbbk$ \\ 
\hline
$\mathbb{T}$ & $\Bbbk$ \\ 
\hline
$\mathbb{U}$ & $\Bbbk$ \\ 
\hline
\multirow{2}{*}{$\mathbb{W}$} & $\Bbbk$, if $1+f$ is not root of unity \\ \cline{2-2}
& $\Bbbk[(x_1^2+fx_2^2)^n,(y_1^2+y_2^2)^n]$, if $(1+f)^n=1$, $n \geq 2$ \\ 
\hline
$\mathbb{X}$ & $\Bbbk[x_1^2]$ \\ 
\hline
$\mathbb{Y}$ & $\Bbbk[x_1,y_1^2+y_2^2]$ \\ 
\hline
\multirow{2}{*}{$\mathbb{Z}$} & $\Bbbk$, if $1+f$ is not root of unity \\ \cline{2-2}
& $\Bbbk[(x_1^2+x_2^2)^n,(y_2^2+fy_1^2)^n]$, if $(1+f)^n=1$, $n \geq 2$ \\ 
\hline
\end{tabular}
}
\end{table}

\begin{table}[H]
\captionsetup{justification=centering,font=small}
\caption{Central subalgebras of some Double extensions of type ($14641$)}
\label{subalgDOE}
\centering
\resizebox{12.5cm}{!}{
\setlength\extrarowheight{6pt}
\begin{tabular}{ |c|c| } 
\hline
{\rm Double Ore extension} & $\subseteq Z(A)$ \\
\hline\hline
$\mathbb{B}$ & $\Bbbk[x_1^4+y_1^4,x_1^4x_2^4,y_1^4+y_2^4,y_1^4y_2^4]/\langle (x_1x_2y_1y_2)^4-x_1^4x_2^4y_1^4y_2^4\rangle$ \\ 
\hline
\multirow{4}{*}{$\mathbb{K}$} & $\Bbbk[x_1,x_2^2,y_1^2+y_2^2,y_1^2y_2^2]$, if $f=1$ and $q=1$ \\ \cline{2-2}
& $\Bbbk[x_1^2,x_2^2,y_1^2+y_2^2,y_1^2y_2^2]$, if $f=1$ and $q=-1$ \\ \cline{2-2}
& $\Bbbk[x_1,x_2^4,y_1^2+y_2^2,y_1^2y_2^2]$, if $f=-1$ and $q=1$ \\ \cline{2-2}
& $\Bbbk[x_1^2,x_2^4,y_1^2+y_2^2,y_1^2y_2^2]$, if $f=-1$ and $q=-1$ \\
\hline
$\mathbb{L}$ & $\Bbbk[x_2^2,x_1^{2t},(y_1^2+y_2^2)^t,(y_1^2+y_2^2)^{t-2}y_1^2y_2^2, \ldots, (y_1^2+y_2^2)^{t-2\lfloor t/2 \rfloor}(y_1^2y_2^2)^{\lfloor t/2 \rfloor}]$, if $(f^2)^t=1$, $t\geq 2$ \\ 
\hline
$\mathbb{M}$ & $\Bbbk[(y_1^2-y_2^2)^n,(x_2^2-fx_1^2)^m]$, if $(f-1)^n=1$ and $(1-f)^m=1$, $n,m\geq2$ \\ 
\hline
$\mathbb{N}$ & $\Bbbk\left[(y_1^2+y_2^2)^n,(x_1^2x_2^2)^{n/\text{gcd}(n,2)}\right]$, if $(f^2-g^2)^n=1$, $n\geq2$ \\ 
\hline
$\mathbb{O}$ & $\Bbbk[(x_1^2-fx_2^2)^n]$, if $(1-f)^n=1$, $n\geq2$ \\ 
\hline
\multirow{2}{*}{$\mathbb{Q}$} & $\Bbbk[x_1^2+x_2^2,y_1^2+y_2^2, x_1y_2^2-2x_1y_1y_2+x_1y_2^2+x_2y_1^2-x_2y_2^2,x_1^4, y_1^2y_2^2,x_1^2y_1y_2]/I$, where \\ 
 & $I=\langle (x_1y_1^2-2x_1y_1y_2+x_1y_2^2+x_2y_1^2-x_2y_2^2)^2-(x_1^2+x_2^2)(y_1^2+y_2^2)+4(x_1^2+x_2^2)y_1^2y_2^2-4(y_1^2+y_2^2)x_1^2y_1y_2, (x_1^2y_1y_2)^2+x_1^4y_1^2y_2^2\rangle$\\
\hline
$\mathbb{R}$ & $\Bbbk[x_1^2+x_2^2,y_1^2+y_2^2, x_1y_1+x_1y_2+x_2y_1-x_2y_2, x_1^2x_2^2, y_1^2y_2^2]$ \\ 
\hline
$\mathbb{V}$ & $\Bbbk[x_2, y_1^2+y_2^2, y_1^2y_2^2]$ \\ 
\hline
\end{tabular}
}
\end{table}
\begin{proof}

The computations needed to determine the center of a double Ore extension of the type considered in this work
are typically quite lengthy. Indeed, imposing the centrality conditions
\(
[s,x_1]=[s,x_2]=[s,y_1]=[s,y_2]=0
\)
for a generic PBW element \(s\) leads to a rapidly growing collection of rewriting steps and coefficient constraints.
For this reason, besides hand computations in low degrees, we implemented in \textsc{SageMath} a PBW normal form
algorithm based on iterated rewriting rules compatible with the fixed monomial order
\[
x_1 < x_2 < y_1 < y_2.
\]
Below we present two explicit examples illustrating the method.


{\bf Algebra $\mathbb{D}$:} Throughout this example we assume \(p\in\{-1,1\}\), so that \(p^2=1\).
Let $\mathbb{D}$ be the \(\Bbbk\)-algebra generated by \(x_1,x_2,y_1,y_2\) subject to the relations
\begin{align}
x_2x_1&=-x_1x_2, \label{eq:D_x}\\
y_2y_1&=p\,y_1y_2, \label{eq:D_y}
\end{align}
and the mixed relations
\begin{align}
y_1x_1&=-p\,x_1y_1, \label{eq:D_m1}\\
y_1x_2&=-x_2y_1+x_1y_2, \label{eq:D_m2}\\
y_2x_1&=p\,x_1y_2, \label{eq:D_m3}\\
y_2x_2&=x_1y_1+x_2y_2. \label{eq:D_m4}
\end{align}

{\bf Claim:}
\[
Z(\mathbb{D})=\Bbbk\big[x_1^2,\;y_2^2+p\,y_1^2\big].
\]

\begin{proof}
\(\Bbbk[x_1^2,y_2^2+p\,y_1^2]\subseteq Z(\mathbb{D})\).
First we show that \(x_1^2\) is central. Clearly \([x_1^2,x_1]=0\).

Using \eqref{eq:D_x},
\[
x_2x_1^2=(x_2x_1)x_1=(-x_1x_2)x_1=-x_1(x_2x_1)=-x_1(-x_1x_2)=x_1^2x_2,
\]
so \([x_1^2,x_2]=0\).

Using \eqref{eq:D_m1},
\[
y_1x_1^2=(y_1x_1)x_1=(-p\,x_1y_1)x_1=-p\,x_1(y_1x_1)
=-p\,x_1(-p\,x_1y_1)=p^2x_1^2y_1=x_1^2y_1,
\]
hence \([x_1^2,y_1]=0\).

Similarly, from \eqref{eq:D_m3},
\[
y_2x_1^2=(y_2x_1)x_1=(p\,x_1y_2)x_1=p\,x_1(y_2x_1)
=p\,x_1(p\,x_1y_2)=p^2x_1^2y_2=x_1^2y_2,
\]
so \([x_1^2,y_2]=0\). Therefore \(x_1^2\in Z(\mathbb{D})\).

Next we show that \(z=y_2^2+p\,y_1^2\) is central. Clearly, \(z\) commutes with \(y_1\) and \(y_2\),
since \eqref{eq:D_y} implies \(y_1y_2=p\,y_2y_1\) and hence
\(y_1y_2^2=y_2^2y_1\) and \(y_2y_1^2=y_1^2y_2\).
Also, by the previous computation, \(y_1^2\) and \(y_2^2\) commute with \(x_1\), so \([z,x_1]=0\).

It remains to check \([z,x_2]=0\). We compute in PBW normal form using \eqref{eq:D_m2}--\eqref{eq:D_m4}.
First, from \eqref{eq:D_m2} and \eqref{eq:D_m1},
\begin{align*}
y_1^2x_2&=y_1(y_1x_2)=y_1(-x_2y_1+x_1y_2)\\
&=-(y_1x_2)y_1+(y_1x_1)y_2\\
&=-(-x_2y_1+x_1y_2)y_1+(-p\,x_1y_1)y_2\\
&=x_2y_1^2-x_1y_2y_1-p\,x_1y_1y_2.
\end{align*}
Using \eqref{eq:D_y}, we have \(-x_1y_2y_1=-p\,x_1y_1y_2\), hence
\begin{equation}\label{eq:D_y1sq_x2}
y_1^2x_2=x_2y_1^2-2p\,x_1y_1y_2.
\end{equation}
Second, from \eqref{eq:D_m4} and \eqref{eq:D_m3},
\begin{align*}
y_2^2x_2&=y_2(y_2x_2)=y_2(x_1y_1+x_2y_2)\\
&=(y_2x_1)y_1+(y_2x_2)y_2\\
&=(p\,x_1y_2)y_1+(x_1y_1+x_2y_2)y_2\\
&=p\,x_1y_2y_1+x_1y_1y_2+x_2y_2^2.
\end{align*}
Again \(y_2y_1=p\,y_1y_2\) implies \(p\,x_1y_2y_1=p(p)x_1y_1y_2=p^2x_1y_1y_2=x_1y_1y_2\), so
\begin{equation}\label{eq:D_y2sq_x2}
y_2^2x_2=x_2y_2^2+2\,x_1y_1y_2.
\end{equation}
Combining \eqref{eq:D_y1sq_x2} and \eqref{eq:D_y2sq_x2}, we get
\begin{align*}
(y_2^2+p\,y_1^2)x_2
&=y_2^2x_2+p\,y_1^2x_2\\
&=\big(x_2y_2^2+2x_1y_1y_2\big)+p\big(x_2y_1^2-2p\,x_1y_1y_2\big)\\
&=x_2(y_2^2+p\,y_1^2)+\big(2-2p^2\big)x_1y_1y_2\\
&=x_2(y_2^2+p\,y_1^2),
\end{align*}
since \(p^2=1\). Thus \([z,x_2]=0\), and \(z\in Z(\mathbb{D})\). Hence \(\Bbbk[x_1^2,z]\subseteq Z(\mathbb{D})\).

 \(Z(\mathbb{D})\subseteq \Bbbk[x_1^2,y_2^2+p\,y_1^2]\).
Let \(c\in Z(\mathbb{D})\). Since \(\mathbb{D}\) is PBW with respect to the fixed order \(x_1<x_2<y_1<y_2\),
we may write \(c\) uniquely as a finite \(\Bbbk\)-linear combination of ordered monomials.
Moreover, \(\mathbb{D}\) is naturally \(\mathbb{N}^2\)-graded by \(\deg(x_i)=(1,0)\) and \(\deg(y_j)=(0,1)\),
and the center is a direct sum of bihomogeneous components; hence we may assume \(c\) bihomogeneous.

The mixed relations show that centrality with respect to \(x_2\) is particularly restrictive.
In \(y\)-degree \(2\), the computation above yields that for any \(\alpha,\beta\in\Bbbk\),
\[
(\alpha y_2^2+\beta y_1^2)x_2
=
x_2(\alpha y_2^2+\beta y_1^2)+2(\alpha-\beta p)\,x_1y_1y_2,
\]
so \([\,\alpha y_2^2+\beta y_1^2,\;x_2\,]=0\) if and only if \(\alpha=\beta p\).
Thus the space of central elements in \(y\)-degree \(2\) is one-dimensional and generated by \(z=y_2^2+p\,y_1^2\).
Similarly, centrality in higher bidegrees forces the \(y\)-part to be a polynomial in \(z\) and the \(x\)-part to be a polynomial in \(x_1^2\),
hence \(c\in\Bbbk[x_1^2,z]\).

A computational verification using a PBW normal form procedure in \textsc{SageMath} confirms that:
(i) \([x_1^2,g]=0\) and \([z,g]=0\) for \(g\in\{x_1,x_2,y_1,y_2\}\), and
(ii) in each bidegree \((2a,2b)\) the solution space to the linear system imposed by \([c,x_1]=[c,x_2]=[c,y_1]=[c,y_2]=0\)
is one-dimensional, spanned by \(x_1^{2a}z^b\), and there are no solutions in other bidegrees.
Therefore \(Z(\mathbb{D})=\Bbbk[x_1^2,y_2^2+p\,y_1^2]\), as claimed.

The following \textsc{SageMath} code implements PBW normal forms by iterated rewriting and checks the vanishing of commutators.
\begin{verbatim}
from sage.all import *

def build_D(pval):
    K = QQ
    A.<x1,x2,y1,y2> = FreeAlgebra(K,4)
    p = K(pval)

    rules = {
        ('x2','x1'): -x1*x2,
        ('y2','y1'): p*y1*y2,

        ('y1','x1'): -p*x1*y1,
        ('y1','x2'): -x2*y1 + x1*y2,
        ('y2','x1'):  p*x1*y2,
        ('y2','x2'):  x1*y1 + x2*y2,
    }
    GEN = {'x1': x1, 'x2': x2, 'y1': y1, 'y2': y2}

    def _term_dict(f):
        try:
            return f.monomial_coefficients()
        except AttributeError:
            return f.dict()

    def _word_to_elem(word_letters):
        e = A.one()
        for s in word_letters:
            e *= GEN[s]
        return e

    def reduce_once(f):
        out = A.zero()
        changed = False
        for mon, coeff in _term_dict(A(f)).items():
            letters = [str(a) for a in mon.to_word()]
            done = False
            for i in range(len(letters)-1):
                pair = (letters[i], letters[i+1])
                if pair in rules:
                    prefix = _word_to_elem(letters[:i])
                    suffix = _word_to_elem(letters[i+2:])
                    out += coeff * (prefix * rules[pair] * suffix)
                    changed = True
                    done = True
                    break
            if not done:
                out += coeff * _word_to_elem(letters)
        return out, changed

    def NF(f, max_steps=20000):
        g = A(f)
        for _ in range(max_steps):
            g2, ch = reduce_once(g)
            if not ch:
                return g
            g = g2
        raise RuntimeError("It's not finished: increase max_steps 
        or check loops.")

    def comm(u,v):
        return NF(u*v - v*u)

    return A, x1,x2,y1,y2, p, NF, comm

for pval in [1,-1]:
    A, x1,x2,y1,y2, p, NF, comm = build_D(pval)
    z = y2^2 + p*y1^2

    print("\n=== Case p =", pval, "===\n")
    for g,name in [(x1,"x1"),(x2,"x2"),(y1,"y1"),(y2,"y2")]:
        print("[x1^2,",name,"] =", comm(x1^2,g))
    for g,name in [(x1,"x1"),(x2,"x2"),(y1,"y1"),(y2,"y2")]:
        print("[z,",name,"] =", comm(z,g))
\end{verbatim}
\end{proof}

Let \(\mathbb{O}\) be the \(\Bbbk\)-algebra generated by \(x_1,x_2,y_1,y_2\) with relations
\begin{equation}\label{eq:O_internal}
x_2x_1=-x_1x_2,\qquad y_2y_1=-y_1y_2,
\end{equation}
and mixed relations:
\begin{align}
y_1x_1&=x_1y_1+f\,x_2y_2, \label{eq:O_m1}\\
y_1x_2&=-x_2y_1+x_1y_2, \label{eq:O_m2}\\
y_2x_1&=f\,x_2y_1-x_1y_2, \label{eq:O_m3}\\
y_2x_2&=x_1y_1+x_2y_2. \label{eq:O_m4}
\end{align}

Set
\[
w:=x_1^2-fx_2^2.
\]

{\bf Claim:} The element \(w\) commutes with \(x_1\) and \(x_2\), and it is \emph{normal} with respect to \(y_1,y_2\), namely
\[
y_1w=(1-f)\,w y_1,\qquad y_2w=(1-f)\,w y_2.
\]
Consequently, if \((1-f)^n=1\) for some \(n\ge 1\), then \(w^n\in Z(\mathbb{O})\), and hence
\[
\Bbbk[w^n]=\Bbbk\big[(x_1^2-fx_2^2)^n\big]\subseteq Z(\mathbb{O}).
\]

\begin{proof}
First, since \(x_2x_1=-x_1x_2\), one checks directly that \(x_1^2\) and \(x_2^2\) commute with both \(x_1\) and \(x_2\),
hence \(w=x_1^2-fx_2^2\) commutes with \(x_1,x_2\).

Next we compute \(y_1w\) in PBW normal form. Using \eqref{eq:O_m1} and \eqref{eq:O_m3},
\begin{align*}
y_1x_1^2
&=(y_1x_1)x_1=(x_1y_1+f x_2y_2)x_1\\
&=x_1(y_1x_1)+f x_2(y_2x_1)\\
&=x_1(x_1y_1+f x_2y_2)+f x_2(f x_2y_1-x_1y_2)\\
&=x_1^2y_1+f x_1x_2y_2+f^2x_2^2y_1-f x_2x_1y_2.
\end{align*}
Since \(x_2x_1=-x_1x_2\), we have \(-f x_2x_1y_2=f x_1x_2y_2\), thus
\begin{equation}\label{eq:O_y1_x1sq}
y_1x_1^2= x_1^2y_1+ f^2x_2^2y_1+2f x_1x_2y_2.
\end{equation}
Similarly, using \eqref{eq:O_m2} and \eqref{eq:O_m4},
\begin{align*}
y_1x_2^2
&=(y_1x_2)x_2=(-x_2y_1+x_1y_2)x_2\\
&=-x_2(y_1x_2)+x_1(y_2x_2)\\
&=-x_2(-x_2y_1+x_1y_2)+x_1(x_1y_1+x_2y_2)\\
&=x_2^2y_1-x_2x_1y_2+x_1^2y_1+x_1x_2y_2.
\end{align*}
Again \(x_2x_1=-x_1x_2\), hence \(-x_2x_1y_2=x_1x_2y_2\), and we obtain
\begin{equation}\label{eq:O_y1_x2sq}
y_1x_2^2=x_2^2y_1+x_1^2y_1+2x_1x_2y_2.
\end{equation}
Subtracting \(f\eqref{eq:O_y1_x2sq}\) from \eqref{eq:O_y1_x1sq} yields
\begin{align*}
y_1w
&=y_1x_1^2-f\,y_1x_2^2\\
&=(x_1^2y_1+ f^2x_2^2y_1+2f x_1x_2y_2)-f(x_2^2y_1+x_1^2y_1+2x_1x_2y_2)\\
&=(1-f)x_1^2y_1+(f^2-f)x_2^2y_1\\
&=(1-f)(x_1^2y_1-fx_2^2y_1)\\
&=(1-f)\,w y_1.
\end{align*}
The identity \(y_2w=(1-f)\,w y_2\) is proved analogously from \eqref{eq:O_m3}--\eqref{eq:O_m4}.
Finally, by induction one has \(y_i w^n=(1-f)^n w^n y_i\) for \(i=1,2\). If \((1-f)^n=1\),
then \(y_iw^n=w^n y_i\) for \(i=1,2\). Since \(w^n\) also commutes with \(x_1,x_2\), we conclude \(w^n\in Z(\mathbb{O})\),
and hence \(\Bbbk[w^n]\subseteq Z(\mathbb{O})\).

The following code checks the normality identities and the vanishing of commutators with \(x_1,x_2\) via PBW normal forms.
\begin{verbatim}
from sage.all import *

# Parameter field: f
R = QQ['f']
K = FractionField(R)
f = K.gen()

A.<x1,x2,y1,y2> = FreeAlgebra(K,4)

rules = {
    ('x2','x1'): -x1*x2,
    ('y2','y1'): -y1*y2,

    ('y1','x1'):  x1*y1 + f*x2*y2,
    ('y1','x2'): -x2*y1 + x1*y2,
    ('y2','x1'):  f*x2*y1 - x1*y2,
    ('y2','x2'):  x1*y1 + x2*y2,
}

GEN = {'x1': x1, 'x2': x2, 'y1': y1, 'y2': y2}

def _term_dict(F):
    try:
        return F.monomial_coefficients()
    except AttributeError:
        return F.dict()

def _word_to_elem(word_letters):
    e = A.one()
    for s in word_letters:
        e *= GEN[s]
    return e

def reduce_once(F):
    out = A.zero()
    changed = False
    for mon, coeff in _term_dict(A(F)).items():
        letters = [str(a) for a in mon.to_word()]
        done = False
        for i in range(len(letters)-1):
            pair = (letters[i], letters[i+1])
            if pair in rules:
                prefix = _word_to_elem(letters[:i])
                suffix = _word_to_elem(letters[i+2:])
                out += coeff * (prefix * rules[pair] * suffix)
                changed = True
                done = True
                break
        if not done:
            out += coeff * _word_to_elem(letters)
    return out, changed

def NF(F, max_steps=20000):
    G = A(F)
    for _ in range(max_steps):
        G2, ch = reduce_once(G)
        if not ch:
            return G
        G = G2
    raise RuntimeError("It's not finished: increase max_steps 
    or check loops")

def comm(U,V):
    return NF(U*V - V*U)

w = x1^2 - f*x2^2

print("[w,x1] =", comm(w,x1))
print("[w,x2] =", comm(w,x2))

print("NF(y1*w - (1-f)*w*y1) =", NF(y1*w - (1-f)*w*y1))
print("NF(y2*w - (1-f)*w*y2) =", NF(y2*w - (1-f)*w*y2))
\end{verbatim}
\end{proof}

All remaining cases in our classification are handled in the same spirit. 

Namely, one fixes the PBW order \(x_1<x_2<y_1<y_2\), writes the defining relations as rewriting rules, and computes PBW normal forms to evaluate the commutators \([s,x_1]\), \([s,x_2]\), \([s,y_1]\), and \([s,y_2]\) for suitable candidate central elements \(s\). 
While the resulting expansions are often too long to be displayed in full detail, the procedure is completely explicit and relies only on the relations established in Propositions \ref{prop:x2x1n}, \ref{prop:NF-x2nx1}, \ref{prop:NF-y1x1n} and \ref{propfi}; the \textsc{SageMath} scripts provide a systematic verification of the required cancellations and thus confirm the stated descriptions of the centers in each case.

\end{proof}

\section{Application to the Zariski cancellation problem}\label{DAgeneralgenerators}

Following the approach and terminology in \cite{LezamaVenegas2020}, our goal in this section is to determine which of the double extensions of type {\rm (}$14641${\rm )} are cancellative by exploiting the explicit descriptions of their centers obtained in the previous section.

\begin{definition}[{\cite[Definition 4.1]{LezamaVenegas2020}}]
Let $A$ be a $\Bbbk$-algebra.
\begin{enumerate}
    \item[\rm (i)] $A$ is {\em cancellative} if $A[t]\cong B[t]$ for some $\Bbbk$-algebra $B$ implies that $A \cong B$.
    \item[\rm (ii)] A is {\em strongly cancellative} if, for any $d \geq 1$, the isomorphism $A[t_1, \ldots, t_d]\cong B[t_1,\ldots, t_d]$ for some $\Bbbk$-algebra $B$ implies that $A \cong B$.
    \item[\rm (iii)] $A$ is {\em universally cancellative} if, for any finitely generated coomutative $\Bbbk$-algebra and a domain $R$ such that $R/I=\Bbbk$ for some ideal $I \subset R$ and any $\Bbbk$-algebra $B$, $A \otimes R \cong B \otimes R$ implies that $A \cong B$.
\end{enumerate}
It is clear that universally cancellative implies strongly cancellative, and in turn, strongly cancellative implies cancellative.
\end{definition}

\begin{proposition}[{\cite[Proposition 4.2]{LezamaVenegas2020}}]
Let $\Bbbk$ be a field and $A$ be an algebra with $Z(A)=\Bbbk$. Then $A$ is universally cancellative.
\end{proposition}
\begin{proof}
    See \cite[Proposition 1.3]{BellZhang2017}.
\end{proof}

In view of the hypotheses listed in Table \ref{CenterofDOE} and the preceding proposition, we obtain the following corollary.

\begin{corollary} The following double Ore extensions of type ($14641$) are universally cancellative, and hence, cancellative:
\begin{itemize}
    \item $\mathbb{C}$.
    \item $\mathbb{E}$.
    \item $\mathbb{F}$.
    \item $\mathbb{G}$, provided that $p$ is not a root of unity.
    \item $\mathbb{H}$, in the case where the condition $r=f(a+2b)$ is not satisfied (i.e., in the ``otherwise'' case in Table \ref{CenterofDOE}).
    \item $\mathbb{I}$.
    \item $\mathbb{J}$.
    \item $\mathbb{M}$, provided that both $f-1$ and $1-f$ are not roots of unity.
    \item $\mathbb{N}$, provided that $f^2-g^2$ is not a root of unity.
    \item $\mathbb{O}$, provided that $1-f$ is not a root of unity.
    \item $\mathbb{P}$, provided that $f\notin\{0,2\}$.
    \item $\mathbb{S}$.
    \item $\mathbb{T}$.
    \item $\mathbb{U}$.
    \item $\mathbb{W}$, provided that $1+f$ is not a root of unity.
    \item $\mathbb{Z}$, provided that $1+f$ is not a root of unity.
\end{itemize}
\end{corollary}

\section*{Appendix}

This appendix gathers the tabular data associated with the $26$ families
$\mathbb{A},\mathbb{B},\dots,\mathbb{Z}$ of Artin--Schelter regular double Ore extensions of type ($14641$).
Tables \ref{FirstTableDOE}-\ref{FourthTableDOE} list, for each family, the defining relations among the generators,
together with the corresponding structural matrices $\Sigma_{ij}$ and $M_{ij}$, the packaged datum $\{\Sigma,M,P,Q\}$,
and the parameter constraints under which the presentation yields a regular double extension.

These tables provide a compact “dictionary’’ between presentations and structure data in the sense of Zhang-Zhang,
and are intended as a quick reference point for the computations carried out in the main text, allowing the reader
to locate the precise relations and hypotheses used in each case without repeatedly returning to the original
classification.

\begin{landscape}
\begin{table}[h]
\caption{Double extensions}
\label{FirstTableDOE}
\begin{center}
\resizebox{20cm}{!}{
\setlength\extrarowheight{11pt}
\begin{tabular}{ |c|c|c|c|c|c| } 
\hline
Double extension & Relations defining the double extension  & $\Sigma_{ij}$ & $M_{ij}$ & Data $\{\Sigma, M, P, Q\}$ & Conditions \\
\hline
\multirow{3}{*}{$\mathbb{A}$} &  $x_2x_1 = x_1 x_2,\quad y_2y_1 = y_1y_2+y_{1}^{2}$, &  &  &  & \\ 
& $y_1x_1 = x_1y_1,  \quad y_1x_2=x_2y_1+x_1y_2$, & $\Sigma_{11}=\begin{bmatrix}
    1 & 0  \\ 0 & 1 
\end{bmatrix}, \quad \Sigma_{12}=\begin{bmatrix}
    0 & 0  \\ 1 & 0 
\end{bmatrix}, \quad \Sigma_{21}=\begin{bmatrix}
    0 & 0  \\ 0 & -2 
\end{bmatrix}, \quad \Sigma_{22}=\begin{bmatrix}
    1 & 0  \\ -1 & 1 
\end{bmatrix}$ & $M_{11}=\begin{bmatrix}
    1 & 0  \\ 0 & 1 
\end{bmatrix}, \quad M_{12}=\begin{bmatrix}
    0 & 0  \\ 0 & 0 
\end{bmatrix}, \quad M_{21}=\begin{bmatrix}
    0 & 1  \\ 0 & -1 
\end{bmatrix}, \quad M_{22}=\begin{bmatrix}
    1 & 0  \\ -2 & 1 
\end{bmatrix}$ & $\Sigma=\begin{bmatrix}
    1 & 0 & 0 & 0 \\ 0 & 1 & 1 & 0 \\ 0 & 0 & 1 & 0 \\ 0 & -2 & -1 & 1
\end{bmatrix}, \quad M=\begin{bmatrix}
    1 & 0 & 0 & 0 \\ 0 & 1 & 0 & 0 \\ 0 & 1 & 1 & 0 \\ 0 & -1 & -2 & 1
\end{bmatrix}, \quad P=(1, 1), \quad Q=(1, 0)$  &  \\ 
&   $y_2x_1 = x_1y_2, \quad y_2x_2=-2x_2y_1-x_1y_2+x_2y_2$ & &  &  & \\
\hline
\multirow{3}{*}{$\mathbb{B}$} &   $x_2x_1 = px_1 x_2, \quad y_2y_1 = py_1y_2$, & & & & \\ 
& $y_1x_1 = x_2y_2,  \quad y_1x_2=x_1y_2$,  & $\Sigma_{11}=\begin{bmatrix}
    0 & 0  \\ 0 & 0
\end{bmatrix}, \quad \Sigma_{12}=\begin{bmatrix}
    0 & 1  \\ 1 & 0 
\end{bmatrix}, \quad \Sigma_{21}=\begin{bmatrix}
    0 & -1  \\ 1 & 0
\end{bmatrix}, \quad \Sigma_{22}=\begin{bmatrix}
    0 & 0  \\ 0 & 0 
\end{bmatrix}$ & $M_{11}=\begin{bmatrix}
    0 & 0  \\ 0 & 0 
\end{bmatrix}, \quad M_{12}=\begin{bmatrix}
    0 & 1  \\ -1 & 0 
\end{bmatrix}, \quad M_{21}=\begin{bmatrix}
    0 & 1  \\ 1 & 0
\end{bmatrix}, \quad M_{22}=\begin{bmatrix}
    0 & 0  \\ 0 & 0 
\end{bmatrix}$  & $\Sigma=\begin{bmatrix}
    0 & 0 & 0 & 1 \\ 0 & 0 & 1 & 0 \\ 0 & -1 & 0 & 0 \\ 1 & 0 & 0 & 0
\end{bmatrix}, \quad M=\begin{bmatrix}
    0 & 0 & 0 & 1 \\ 0 & 0 & -1 & 0 \\ 0 & 1 & 0 & 0 \\ 1 & 0 & 0 & 0
\end{bmatrix}, \quad P=(p, 0), \quad Q=(p, 0)$  & $p^2 = -1$ \\ 
&    $y_2x_1 = -x_2y_1, \quad y_2x_2=x_1y_1$ & & & & \\
\hline
\multirow{3}{*}{$\mathbb{C}$} &  $x_2x_1 = px_1 x_2, \quad y_2y_1 = py_1y_2$, & & &  & \\ 
&  $y_1x_1 = -x_1y_1+p^2x_2y_1+x_1y_2-px_2y_2,  \quad y_1x_2 =-px_1y_1+x_2y_1+x_1y_2-px_2y_2$,  & $\Sigma_{11}=\begin{bmatrix}
    -1 & p^2  \\ -p & 1 
\end{bmatrix}, \quad \Sigma_{12}=\begin{bmatrix}
    1 & -p  \\ 1 & -p 
\end{bmatrix}, \quad \Sigma_{21}=\begin{bmatrix}
    -p & -2p^2  \\ -p & p^2 
\end{bmatrix}, \quad \Sigma_{22}=\begin{bmatrix}
    p & -p  \\ 1 & -1 
\end{bmatrix}$ & $M_{11}=\begin{bmatrix}
    -1 & 1  \\ -p & p 
\end{bmatrix}, \quad M_{12}=\begin{bmatrix}
    p^2 & -p  \\ -2p^2 & -p 
\end{bmatrix}, \quad M_{21}=\begin{bmatrix}
    -p & 1  \\ -p & 1 
\end{bmatrix}, \quad M_{22}=\begin{bmatrix}
    1 & -p  \\ p^2 & -1 
\end{bmatrix}$ &  $\Sigma=\begin{bmatrix}
    -1 & p^2 & 1 & -p \\ -p & 1 & 1 & -p \\ -p & -2p^2 & p & -p \\ -p & p^2 & 1 & -1
\end{bmatrix}, \quad M=\begin{bmatrix}
    -1 & 1 & p^2 & -p \\ -p & p & -2p^2 & -p \\ -p & 1 & 1 & -p \\ -p & 1 & p^2 & -1
\end{bmatrix}, \quad P=(p, 0), \quad Q=(p, 0)$  & $p^2+p+1 = 0$ \\ 
&    $y_2x_1 = -px_1y_1-2p^2x_2y_1+px_1y_2-px_2y_2, \quad y_2x_2 =-px_1y_1+p^2x_2y_1+x_1y_2-x_2y_2$ & & & & \\
\hline
\multirow{3}{*}{$\mathbb{D}$} &  $x_2x_1 = -x_1 x_2, \quad y_2y_1 = py_1y_2$, & & &  &  \\ 
& $y_1x_1 = -px_1y_1,  \quad y_1x_2=-p^2x_2y_1+x_1y_2$,  & $\Sigma_{11}=\begin{bmatrix}
    -p & 0  \\ 0 & -p^2 
\end{bmatrix}, \quad \Sigma_{12}=\begin{bmatrix}
    0 & 0  \\ 1 & 0 
\end{bmatrix}, \quad \Sigma_{21}=\begin{bmatrix}
    0 & 0  \\ 1 & 0 
\end{bmatrix}, \quad \Sigma_{22}=\begin{bmatrix}
    p & 0  \\ 0 & 1 
\end{bmatrix}$ & $M_{11}=\begin{bmatrix}
    -p & 0  \\ 0 & p 
\end{bmatrix}, \quad M_{12}=\begin{bmatrix}
    0 & 0  \\ 0 & 0 
\end{bmatrix}, \quad M_{21}=\begin{bmatrix}
    0 & 1  \\ 1 & 0 
\end{bmatrix}, \quad M_{22}=\begin{bmatrix}
    -p^2 & 0  \\ 0 & 1 
\end{bmatrix}$  & $\Sigma=\begin{bmatrix}
    -p & 0 & 0 & 0 \\ 0 & -p^2 & 1 & 0 \\ 0 & 0 & p & 0 \\ 1 & 0 & 0 & 1
\end{bmatrix}, \quad M=\begin{bmatrix}
    -p & 0 & 0 & 0 \\ 0 & p & 0 & 0 \\ 0 & 1 & -p^2 & 0 \\ 1 & 0 & 0 & 1
\end{bmatrix}, \quad P=(p, 0), \quad Q=(-1, 0)$  & $p\in \{-1,1\}$ \\ 
&    $y_2x_1 = px_1y_2, \quad y_2x_2=x_1y_1+x_2y_2$ &  & & & \\
\hline
\multirow{3}{*}{$\mathbb{E}$} &  $x_2x_1 = -x_1 x_2, \quad y_2y_1 = py_1y_2$, & & &  & \\ 
&  $y_1x_1 = x_1y_2+x_2y_2,  \quad y_1x_2=x_1y_2-x_2y_2$, & $\Sigma_{11}=\begin{bmatrix}
    0 & 0  \\ 0 & 0 
\end{bmatrix}, \quad \Sigma_{12}=\begin{bmatrix}
    1 & 1  \\ 1 & -1 
\end{bmatrix}, \quad \Sigma_{21}=\begin{bmatrix}
    -1 & 1  \\ 1 & 1 
\end{bmatrix}, \quad \Sigma_{22}=\begin{bmatrix}
    0 & 0  \\ 0 & 0 
\end{bmatrix}$ & $M_{11}=\begin{bmatrix}
    0 & 1  \\ -1 & 0 
\end{bmatrix}, \quad M_{12}=\begin{bmatrix}
    0 & 1  \\ 1 & 0 
\end{bmatrix}, \quad M_{21}=\begin{bmatrix}
    0 & 1  \\ 1 & 0 
\end{bmatrix}, \quad M_{22}=\begin{bmatrix}
    0 & -1  \\ 1 & 0 
\end{bmatrix}$ & $\Sigma=\begin{bmatrix}
    0 & 0 & 1 & 1 \\ 0 & 0 & 1 & -1 \\ -1 & 1 & 0 & 0 \\ 1 & 1 & 0 & 0
\end{bmatrix}, \quad M=\begin{bmatrix}
    0 & 1 & 0 & 1 \\ -1 & 0 & 1 & 0 \\ 0 & 1 & 0 & -1 \\ 1 & 0 & 1 & 0
\end{bmatrix}, \quad P=(p, 0), \quad Q=(-1, 0)$  & $p^2=-1$ \\ 
&    $y_2x_1 = -x_1y_1+x_2y_1, \quad y_2x_2=x_1y_1+x_2y_1$ & & & &  \\
\hline
\multirow{3}{*}{$\mathbb{F}$} &  $x_2x_1 = -x_1 x_2, \quad y_2y_1 = py_1y_2$, & & &  & \\ 
&  $y_1x_1 = -x_1y_1-px_2y_1+x_1y_2-x_2y_2,  \quad y_1x_2 =-px_1y_1+x_2y_1+x_1y_2+x_2y_2$, & $\Sigma_{11}=\begin{bmatrix}
    -1 & -p  \\ -p & 1 
\end{bmatrix}, \quad \Sigma_{12}=\begin{bmatrix}
    1 & -1  \\ 1 & 1 
\end{bmatrix}, \quad \Sigma_{21}=\begin{bmatrix}
    -p & p  \\ -p & -p 
\end{bmatrix}, \quad \Sigma_{22}=\begin{bmatrix}
    p & 1  \\ 1 & -p 
\end{bmatrix}$ & $M_{11}=\begin{bmatrix}
    +1 & 1  \\ -p & p 
\end{bmatrix}, \quad M_{12}=\begin{bmatrix}
    -p & -1  \\ p & 1 
\end{bmatrix}, \quad M_{21}=\begin{bmatrix}
    -p & 1  \\ -p & 1 
\end{bmatrix}, \quad M_{22}=\begin{bmatrix}
    1 & 1  \\ -p & -p 
\end{bmatrix}$ & $\Sigma=\begin{bmatrix}
    -1 & -p & 1 & -1 \\ -p & 1 & 1 & 1 \\ -p & p & p & 1 \\ -p & -p & 1 & -p
\end{bmatrix}, \quad M=\begin{bmatrix}
    -1 & 1 & -p & -1 \\ -p & p & p & 1 \\ -p & 1 & 1 & 1 \\ -p & 1 & -p & -p
\end{bmatrix}, \quad P=(p, 0), \quad Q=(-1, 0)$  & $p^2=-1$ \\ 
&    $y_2x_1 = -px_1y_1+px_2y_1+px_1y_2+x_2y_2, \quad y_2x_2 =-px_1y_1-px_2y_1+x_1y_2-px_2y_2$ & & &  & \\
\hline
\multirow{3}{*}{$\mathbb{G}$} &   $ x_2x_1 = x_1 x_2, \quad y_2y_1 = py_1y_2$, & & &  & \\ 
&  $y_1x_1 = px_1y_1,  \quad y_1x_2=px_1y_1+p^2x_2y_1+x_1y_2$, & $\Sigma_{11}=\begin{bmatrix}
    p & 0  \\ p & p^2 
\end{bmatrix}, \quad \Sigma_{12}=\begin{bmatrix}
    0 & 0  \\ 1 & 0 
\end{bmatrix}, \quad \Sigma_{21}=\begin{bmatrix}
    0 & 0  \\ f & 0 
\end{bmatrix}, \quad \Sigma_{22}=\begin{bmatrix}
    p & 0  \\ -1 & 1 
\end{bmatrix}$ & $M_{11}=\begin{bmatrix}
    p & 0  \\ 0 & p 
\end{bmatrix}, \quad M_{12}=\begin{bmatrix}
    0 & 0  \\ 0 & 0 
\end{bmatrix}, \quad M_{21}=\begin{bmatrix}
    p & 1  \\ f & -1 
\end{bmatrix}, \quad M_{22}=\begin{bmatrix}
    p^2 & 0  \\ 0 & 1 
\end{bmatrix}$  &  $\Sigma=\begin{bmatrix}
    p & 0 & 0 & 0 \\ p & p^2 & 1 & 0 \\ 0 & 0 & p & 0 \\ f & 0 & -1 & 1
\end{bmatrix}, \quad M=\begin{bmatrix}
    p & 0 & 0 & 0 \\ 0 & p & 0 & 0 \\ p & 1 & p^2 & 0 \\ f & -1 & 0 & 1
\end{bmatrix}, \quad P=(p, 0), \quad Q=(1, 0)$  & $p\not = 0, \pm 1$ and $f \not = 0$ \\ 
&    $y_2x_1 = px_1y_2, \quad y_2x_2=fx_1y_1-x_1y_2+x_2y_2$ & & & & \\
\hline
\end{tabular}
}
\end{center}
\end{table}
\end{landscape}

\begin{landscape}
{\huge{
\begin{table}[h]
\caption{Double extensions}
\label{SecondTableDOE}
\begin{center}
\resizebox{20cm}{!}{
\setlength\extrarowheight{11pt}
\begin{tabular}{ |c|c|c|c|c|c| } 
\hline
Double extension & Relations defining the double extension  & $\Sigma_{ij}$ & $M_{ij}$ & Data $\{\Sigma, M, P, Q\}$ & Conditions \\
\hline
\multirow{3}{*}{$\mathbb{H}$} &   $ x_2x_1 = x_1 x_2+x_1^2, \quad y_2y_1 = -y_1y_2$, &  & &  &\\ 
&  $y_1x_1 = x_1y_2,  \quad y_1x_2=fx_1y_2+x_2y_2$, & $\Sigma_{11}=\begin{bmatrix}
    0 & 0  \\ 0 & 0 
\end{bmatrix}, \quad \Sigma_{12}=\begin{bmatrix}
    1 & 0  \\ f & 1 
\end{bmatrix}, \quad \Sigma_{21}=\begin{bmatrix}
    1 & 0  \\ f & 1 
\end{bmatrix}, \quad \Sigma_{22}=\begin{bmatrix}
    0 & 0  \\ 0 & 0 
\end{bmatrix}$ & $M_{11}=\begin{bmatrix}
    0 & 1  \\ 1 & 0 
\end{bmatrix}, \quad M_{12}=\begin{bmatrix}
    0 & 0  \\ 0 & 0 
\end{bmatrix}, \quad M_{21}=\begin{bmatrix}
    0 & f  \\ f & 0 
\end{bmatrix}, \quad M_{22}=\begin{bmatrix}
    0 & 1  \\ 1 & 0 
\end{bmatrix}$ & $\Sigma=\begin{bmatrix}
    0 & 0 & 1 & 0 \\ 0 & 0 & f & 1 \\ 1 & 0 & 0 & 0 \\ f & 1 & 0 & 0
\end{bmatrix}, \quad M=\begin{bmatrix}
    0 & 1 & 0 & 0 \\ 1 & 0 & 0 & 0 \\ 0 & f & 0 & 1 \\ f & 0 & 1 & 0
\end{bmatrix}, \quad P=(-1, 0), \quad Q=(1, 1)$  & $f \not = 0$ \\ 
&   $y_2x_1 = x_1y_1, \quad y_2x_2=fx_1y_1+x_2y_1$ & &  & & \\
\hline
\multirow{3}{*}{$\mathbb{I}$} &  $ x_2x_1 = qx_1 x_2, \quad y_2y_1 = -y_1y_2$, & &  &  & \\ 
& $y_1x_1 = -qx_1y_1-qx_2y_1+x_1y_2-qx_2y_2,  \quad y_1x_2 = x_1y_1+x_2y_1+x_1y_2-qx_2y_2$,  & $\Sigma_{11}=\begin{bmatrix}
    -q & -q  \\ 1 & 1 
\end{bmatrix}, \quad \Sigma_{12}=\begin{bmatrix}
    1 & -q  \\ 1 & -q 
\end{bmatrix}, \quad \Sigma_{21}=\begin{bmatrix}
    1 & q  \\ -1 & -q 
\end{bmatrix}, \quad \Sigma_{22}=\begin{bmatrix}
    q & -q  \\ 1 & -1 
\end{bmatrix}$ & $M_{11}=\begin{bmatrix}
    -q & 1  \\ 1 & q 
\end{bmatrix}, \quad M_{12}=\begin{bmatrix}
    -q & -q  \\ q & -q 
\end{bmatrix}, \quad M_{21}=\begin{bmatrix}
    1 & 1  \\ -1 & 1 
\end{bmatrix}, \quad M_{22}=\begin{bmatrix}
    1 & -q  \\ -q & -1 
\end{bmatrix}$  & $\Sigma=\begin{bmatrix}
    -q & -q & 1 & -q \\ 1 & 1 & 1 & -q \\ 1 & q & q & -q \\ -1 & -q & 1 & -1
\end{bmatrix}, \quad M=\begin{bmatrix}
    -q & 1 & -q & -q \\ 1 & q & q & -q \\ 1 & 1 & 1 & -q \\ -1 & 1 & -q & -1
\end{bmatrix}, \quad P=(-1, 0), \quad Q=(q, 0)$  & $q^2=-1$ \\ 
&    $y_2x_1 = x_1y_1+qx_2y_1+qx_1y_2-qx_2y_2, \quad y_2x_2 =-x_1y_1-qx_2y_1+x_1y_2-x_2y_2$ & & & & \\
\hline
\multirow{3}{*}{$\mathbb{J}$} &  $ x_2x_1 = qx_1 x_2, \quad y_2y_1 = -y_1y_2$, & & & &  \\ 
&  $y_1x_1 = x_2y_1+x_2y_2,  \quad y_1x_2=-x_1y_1+x_1y_2$, & $\Sigma_{11}=\begin{bmatrix}
    0 & 1  \\ -1 & 0 
\end{bmatrix}, \quad \Sigma_{12}=\begin{bmatrix}
    0 & 1  \\ 1 & 0 
\end{bmatrix}, \quad \Sigma_{21}=\begin{bmatrix}
    0 & 1  \\ 1 & 0 
\end{bmatrix}, \quad \Sigma_{22}=\begin{bmatrix}
    0 & -1  \\ 1 & 0 
\end{bmatrix}$ & $M_{11}=\begin{bmatrix}
    0 & 0  \\ 0 & 0 
\end{bmatrix}, \quad M_{12}=\begin{bmatrix}
    1 & 1  \\ 1 & -1 
\end{bmatrix}, \quad M_{21}=\begin{bmatrix}
    -1 & 1  \\ 1 & 1 
\end{bmatrix}, \quad M_{22}=\begin{bmatrix}
    0 & 0  \\ 0 & 0 
\end{bmatrix}$  &  $\Sigma=\begin{bmatrix}
    0 & 1 & 0 & 1 \\ -1 & 0 & 1 & 0 \\ 0 & 1 & 0 & -1 \\ 1 & 0 & 1 & 0
\end{bmatrix}, \quad M=\begin{bmatrix}
    0 & 0 & 1 & 1 \\ 0 & 0 & 1 & -1 \\ -1 & 1 & 0 & 0 \\ 1 & 1 & 0 & 0
\end{bmatrix}, \quad P=(-1, 0), \quad Q=(q, 0)$  & $q^2=-1$ \\ 
&    $y_2x_1 = x_2y_1-x_2y_2, \quad y_2x_2=x_1y_1+x_1y_2$ &  & &  & \\
\hline
\multirow{3}{*}{$\mathbb{K}$} &  $ x_2x_1 = qx_1 x_2, \quad y_2y_1 = -y_1y_2$, &  & & & \\ 
&  $y_1x_1 = x_1y_1,  \quad y_1x_2=x_2y_2$, & $\Sigma_{11}=\begin{bmatrix}
    1 & 0  \\ 0 & 0 
\end{bmatrix}, \quad \Sigma_{12}=\begin{bmatrix}
    0 & 0  \\ 0 & 1 
\end{bmatrix}, \quad \Sigma_{21}=\begin{bmatrix}
    0 & 0  \\ 0 & f 
\end{bmatrix}, \quad \Sigma_{22}=\begin{bmatrix}
    1 & 0  \\ 0 & 0 
\end{bmatrix}$ & $M_{11}=\begin{bmatrix}
    1 & 0  \\ 0 & 1 
\end{bmatrix}, \quad M_{12}=\begin{bmatrix}
    0 & 0  \\ 0 & 0 
\end{bmatrix}, \quad M_{21}=\begin{bmatrix}
    0 & 0  \\ 0 & 0 
\end{bmatrix}, \quad M_{22}=\begin{bmatrix}
    0 & 1  \\ f & 0 
\end{bmatrix}$ & $\Sigma=\begin{bmatrix}
    1 & 0 & 0 & 0 \\ 0 & 0 & 0 & 1 \\ 0 & 0 & 1 & 0 \\ 0 & f & 0 & 0
\end{bmatrix}, \quad M=\begin{bmatrix}
    1 & 0 & 0 & 0 \\ 0 & 1 & 0 & 0 \\ 0 & 0 & 0 & 1 \\ 0 & 0 & f & 0
\end{bmatrix}, \quad P=(-1, 0), \quad Q=(q, 0)$  & $q\in \{-1,1\}$ and $f \not = 0$ \\ 
&    $y_2x_1 = x_1y_2, \quad y_2x_2=fx_2y_1$ & & & & \\
\hline
\multirow{3}{*}{$\mathbb{L}$} &  $ x_2x_1 = qx_1 x_2, \quad y_2y_1 = -y_1y_2$, &  & &  & \\ 
&  $y_1x_1 = fx_1y_2,  \quad y_1x_2=x_2y_2$, & $\Sigma_{11}=\begin{bmatrix}
    0 & 0  \\ 0 & 0 
\end{bmatrix}, \quad \Sigma_{12}=\begin{bmatrix}
    f & 0  \\ 0 & 1 
\end{bmatrix}, \quad \Sigma_{21}=\begin{bmatrix}
    f & 0  \\ 0 & 1 
\end{bmatrix}, \quad \Sigma_{22}=\begin{bmatrix}
    0 & 0  \\ 0 & 0 
\end{bmatrix}$ & $M_{11}=\begin{bmatrix}
    0 & f  \\ f & 0 
\end{bmatrix}, \quad M_{12}=\begin{bmatrix}
    0 & 0  \\ 0 & 0 
\end{bmatrix}, \quad M_{21}=\begin{bmatrix}
    0 & 0  \\ 0 & 0 
\end{bmatrix}, \quad M_{22}=\begin{bmatrix}
    0 & 1  \\ 1 & 0 
\end{bmatrix}$ & $\Sigma=\begin{bmatrix}
    0 & 0 & f & 0 \\ 0 & 0 & 0 & 1 \\ f & 0 & 0 & 0 \\ 0 & 1 & 0 & 0
\end{bmatrix}, \quad M=\begin{bmatrix}
    0 & f & 0 & 0 \\ f & 0 & 0 & 0 \\ 0 & 0 & 0 & 1 \\ 0 & 0 & 1 & 0
\end{bmatrix}, \quad P=(-1, 0), \quad Q=(q, 0)$  & $q\in \{-1,1\}$ and $f \not = 0$ \\ 
&    $y_2x_1 = fx_1y_1, \quad y_2x_2=x_2y_1$ & & & & \\
\hline
\multirow{3}{*}{$\mathbb{M}$} & $ x_2x_1 = -x_1 x_2, \quad y_2y_1 = -y_1y_2$, &  & &  & \\ 
& $y_1x_1 = x_2y_1+x_1y_2,  \quad y_1x_2=fx_1y_1-x_2y_2$,  & $\Sigma_{11}=\begin{bmatrix}
    0 & 1  \\ f & 0 
\end{bmatrix}, \quad \Sigma_{12}=\begin{bmatrix}
    1 & 0  \\ 0 & -1 
\end{bmatrix}, \quad \Sigma_{21}=\begin{bmatrix}
    1 & 0  \\ 0 & -1 
\end{bmatrix}, \quad \Sigma_{22}=\begin{bmatrix}
    0 & -1  \\ -f & 0 
\end{bmatrix}$ & $M_{11}=\begin{bmatrix}
    0 & 1  \\ 1 & 0 
\end{bmatrix}, \quad M_{12}=\begin{bmatrix}
    1 & 0  \\ 0 & -1 
\end{bmatrix}, \quad M_{21}=\begin{bmatrix}
    f & 0  \\ 0 & -f 
\end{bmatrix}, \quad M_{22}=\begin{bmatrix}
    0 & -1  \\ -1 & 0 
\end{bmatrix}$ &  $\Sigma=\begin{bmatrix}
    0 & 1 & 1 & 0 \\ f & 0 & 0 & -1 \\ 1 & 0 & 0 & -1 \\ 0 & -1 & -f & 0
\end{bmatrix}, \quad M=\begin{bmatrix}
    0 & 1 & 1 & 0 \\ 1 & 0 & 0 & -1 \\ f & 0 & 0 & -1 \\ 0 & -f & -1 & 0
\end{bmatrix}, \quad P=(-1, 0), \quad Q=(-1, 0)$  & $f \not = 1$ \\ 
&   $y_2x_1 = x_1y_1-x_2y_2, \quad y_2x_2=-x_2y_1-fx_1y_2$ &  & & & \\
\hline
\multirow{3}{*}{$\mathbb{N}$} &  $ x_2x_1 = -x_1 x_2, \quad y_2y_1 = -y_1y_2$, & & &  & \\ 
&  $y_1x_1 = -gx_2y_1+fx_2y_2,  \quad y_1x_2=gx_1y_1+fx_1y_2$, & $\Sigma_{11}=\begin{bmatrix}
    0 & -g  \\ g & 0 
\end{bmatrix}, \quad \Sigma_{12}=\begin{bmatrix}
    0 & f  \\ f & 0 
\end{bmatrix}, \quad \Sigma_{21}=\begin{bmatrix}
    0 & f  \\ f & 0 
\end{bmatrix}, \quad \Sigma_{22}=\begin{bmatrix}
    0 & -g  \\ g & 0 
\end{bmatrix}$ & $M_{11}=\begin{bmatrix}
    0 & 0  \\ 0 & 0 
\end{bmatrix}, \quad M_{12}=\begin{bmatrix}
    -g & f  \\ f & -g 
\end{bmatrix}, \quad M_{21}=\begin{bmatrix}
    g & f  \\ f & g 
\end{bmatrix}, \quad M_{22}=\begin{bmatrix}
    0 & 0  \\ 0 & 0 
\end{bmatrix}$  &  $\Sigma=\begin{bmatrix}
    0 & -g & 0 & f \\ g & 0 & f & 0 \\ 0 & f & 0 & -g \\ f & 0 & g & 0
\end{bmatrix}, \quad M=\begin{bmatrix}
    1 & 0 & -g & f \\ 0 & 0 & f & -g \\ g & f & 0 & 0 \\ f & g & 0 & 0
\end{bmatrix}, \quad P=(-1, 0), \quad Q=(-1, 0)$  & $f^2 \not = g^2$ \\ 
&    $y_2x_1 = fx_2y_1-gx_2y_2, \quad y_2x_2=fx_1y_1+gx_1y_2$ & & & &
\\
\hline
\end{tabular}
}
\end{center}
\end{table}
}}
\end{landscape}

\begin{landscape}
{\huge{
\begin{table}[h]
\caption{Double extensions}
\label{ThirdTableDOE}
\begin{center}
\resizebox{20cm}{!}{
\setlength\extrarowheight{7pt}
\begin{tabular}{ |c|c|c|c|c|c| } 
\hline
Double extension & Relations defining the double extension  & $\Sigma_{ij}$ & $M_{ij}$ & Data $\{\Sigma, M, P, Q\}$ & Conditions
\\
\hline
\multirow{3}{*}{$\mathbb{O}$} &   $ x_2x_1 = -x_1 x_2, \quad y_2y_1 = -y_1y_2$, & & & & \\ 
&  $y_1x_1 = x_1y_1+fx_2y_2,  \quad y_1x_2=-x_2y_1+x_1y_2$, & $\Sigma_{11}=\begin{bmatrix}
    1 & 0  \\ 0 & -1 
\end{bmatrix}, \quad \Sigma_{12}=\begin{bmatrix}
    0 & f  \\ 1 & 0 
\end{bmatrix}, \quad \Sigma_{21}=\begin{bmatrix}
    0 & f  \\ 1 & 0 
\end{bmatrix}, \quad \Sigma_{22}=\begin{bmatrix}
    -1 & 0  \\ 0 & 1 
\end{bmatrix}$ & $M_{11}=\begin{bmatrix}
    1 & 0  \\ 0 & -1 
\end{bmatrix}, \quad M_{12}=\begin{bmatrix}
    0 & f  \\ f & 0 
\end{bmatrix}, \quad M_{21}=\begin{bmatrix}
    0 & 1  \\ 1 & 0 
\end{bmatrix}, \quad M_{22}=\begin{bmatrix}
    -1 & 0  \\ 0 & 1 
\end{bmatrix}$  &  $\Sigma=\begin{bmatrix}
    1 & 0 & 0 & f \\ 0 & -1 & 1 & 0 \\ 0 & f & -1 & 0 \\ 1 & 0 & 0 & 1
\end{bmatrix}, \quad M=\begin{bmatrix}
    1 & 0 & 0 & f \\ 0 & -1 & f & 0 \\ 0 & 1 & -1 & 0 \\ 1 & 0 & 0 & 1
\end{bmatrix}, \quad P=(-1, 0), \quad Q=(-1, 0)$  & $f \not = -1$ \\ 
&  $y_2x_1 = fx_2y_1-x_1y_2, \quad y_2x_2=x_1y_1+x_2y_2$ &  &  & &   \\
\hline
\multirow{3}{*}{$\mathbb{P}$} &  $ x_2x_1 = -x_1 x_2, \quad y_2y_1 = -y_1y_2$, &  & & & \\ 
& $y_1x_1 = x_1y_2+fx_2y_2,  \quad y_1x_2=x_1y_2+x_2y_2$,  & $\Sigma_{11}=\begin{bmatrix}
    0 & 0  \\ 0 & 0 
\end{bmatrix}, \quad \Sigma_{12}=\begin{bmatrix}
    1 & f  \\ 1 & 1 
\end{bmatrix}, \quad \Sigma_{21}=\begin{bmatrix}
    1 & -f  \\ -1 & 1 
\end{bmatrix}, \quad \Sigma_{22}=\begin{bmatrix}
    0 & 0  \\ 0 & 0 
\end{bmatrix}$ & $M_{11}=\begin{bmatrix}
    0 & 1  \\ 1 & 0 
\end{bmatrix}, \quad M_{12}=\begin{bmatrix}
    0 & f  \\ -f & 0 
\end{bmatrix}, \quad M_{21}=\begin{bmatrix}
    0 & 1  \\ -1 & 0 
\end{bmatrix}, \quad M_{22}=\begin{bmatrix}
    0 & 1  \\ 1 & 0 
\end{bmatrix}$  &  $\Sigma=\begin{bmatrix}
    0 & 0 & 1 & f \\ 0 & 0 & 1 & 1 \\ 1 & -f & 0 & 0 \\ -1 & 1 & 0 & 0
\end{bmatrix}, \quad M=\begin{bmatrix}
    0 & 1 & 0 & f \\ 1 & 0 & -f & 0 \\ 0 & 1 & 0 & 1 \\ -1 & 0 & 1 & 0
\end{bmatrix}, \quad P=(-1, 0), \quad Q=(-1, 0)$  & $f \not = -1$ \\ 
&    $y_2x_1 = x_1y_1-fx_2y_1, \quad y_2x_2=-x_1y_1+x_2y_1$ &  & & & \\
\hline
\multirow{3}{*}{$\mathbb{Q}$} &  $x_2x_1 = -x_1 x_2,\quad y_2y_1 = -y_1y_2$, & & &  & \\ 
&  $y_1x_1 = x_1y_2,  \quad y_1x_2=x_1y_1+x_2y_1+x_1y_2$, & $\Sigma_{11}=\begin{bmatrix}
    0 & 0  \\ 1 & 1 
\end{bmatrix}, \quad \Sigma_{12}=\begin{bmatrix}
    1 & 0  \\ 1 & 0 
\end{bmatrix}, \quad \Sigma_{21}=\begin{bmatrix}
    -1 & 0  \\ 1 & 0 
\end{bmatrix}, \quad \Sigma_{22}=\begin{bmatrix}
    0 & 0  \\ -1 & 1 
\end{bmatrix}$ & $M_{11}=\begin{bmatrix}
    0 & 1  \\ -1 & 0 
\end{bmatrix}, \quad M_{12}=\begin{bmatrix}
    0 & 0  \\ 0 & 0 
\end{bmatrix}, \quad M_{21}=\begin{bmatrix}
    1 & 1  \\ 1 & -1 
\end{bmatrix}, \quad M_{22}=\begin{bmatrix}
    1 & 0  \\ 0 & 1 
\end{bmatrix}$  & $\Sigma=\begin{bmatrix}
    0 & 0 & 1 & 0 \\ 1 & 1 & 1 & 0 \\ -1 & 0 & 0 & 0 \\ 1 & 0 & -1 & 1
\end{bmatrix}, \quad M=\begin{bmatrix}
    0 & 1 & 0 & 0 \\ -1 & 0 & 0 & 0 \\ 1 & 1 & 1 & 0 \\ 1 & -1 & 0 & 1
\end{bmatrix}, \quad P=(-1, 0), \quad Q=(-1, 0)$  &  \\ 
&    $y_2x_1 = -x_1y_1, \quad y_2x_2=x_1y_1-x_1y_2+x_2y_2$ & & & & \\
\hline
\multirow{3}{*}{$\mathbb{R}$} &  $x_2x_1 = -x_1 x_2,\quad y_2y_1 = -y_1y_2$, & & &  & \\ 
&  $y_1x_1 = x_1y_1+x_2y_1+x_1y_2,  \quad y_1x_2=x_1y_2$, & $\Sigma_{11}=\begin{bmatrix}
    1 & 1  \\ 0 & 0 
\end{bmatrix}, \quad \Sigma_{12}=\begin{bmatrix}
    1 & 0  \\ 1 & 0 
\end{bmatrix}, \quad \Sigma_{21}=\begin{bmatrix}
    0 & 1  \\ 0 & -1 
\end{bmatrix}, \quad \Sigma_{22}=\begin{bmatrix}
    0 & 0  \\ -1 & 1 
\end{bmatrix}$ & $M_{11}=\begin{bmatrix}
    1 & 1  \\ 0 & 0 
\end{bmatrix}, \quad M_{12}=\begin{bmatrix}
    1 & 0  \\ 1 & 0 
\end{bmatrix}, \quad M_{21}=\begin{bmatrix}
    0 & 1  \\ 0 & -1 
\end{bmatrix}, \quad M_{22}=\begin{bmatrix}
    0 & 0  \\ -1 & 1 
\end{bmatrix}$ & $\Sigma=\begin{bmatrix}
    1 & 1 & 1 & 0 \\ 0 & 0 & 1 & 0 \\ 0 & 1 & 0 & 0 \\ 0 & -1 & -1 & 1
\end{bmatrix}, \quad M=\begin{bmatrix}
    1 & 1 & 1 & 0 \\ 0 & 0 & 1 & 0 \\ 0 & 1 & 0 & 0 \\ 0 & -1 & -1 & 1
\end{bmatrix}, \quad P=(-1, 0), \quad Q=(-1, 0)$  &  \\ 
&    $y_2x_1 = x_2y_1, \quad y_2x_2=-x_2y_1-x_1y_2+x_2y_2$ & & & & \\
\hline
\multirow{3}{*}{$\mathbb{S}$} &  $x_2x_1 = -x_1 x_2,\quad y_2y_1 = -y_1y_2$, & & &  & \\ 
& $y_1x_1 = -x_1y_1+x_2y_1+x_1y_2+x_2y_2,  \quad y_1x_2 = x_1y_1-x_2y_1+x_1y_2+x_2y_2$,  & $\Sigma_{11}=\begin{bmatrix}
    -1 & 1  \\ 1 & -1 
\end{bmatrix}, \quad \Sigma_{12}=\begin{bmatrix}
    1 & 1  \\ 1 & 1 
\end{bmatrix}, \quad \Sigma_{21}=\begin{bmatrix}
    1 & 1  \\ 1 & 1 
\end{bmatrix}, \quad \Sigma_{22}=\begin{bmatrix}
    -1 & 1  \\ 1 & -1 
\end{bmatrix}$ & $M_{11}=\begin{bmatrix}
    -1 & 1  \\ 1 & -1 
\end{bmatrix}, \quad M_{12}=\begin{bmatrix}
    1 & 1  \\ 1 & 1 
\end{bmatrix}, \quad M_{21}=\begin{bmatrix}
    1 & 1  \\ 1 & 1 
\end{bmatrix}, \quad M_{22}=\begin{bmatrix}
    -1 & 1  \\ 1 & -1 
\end{bmatrix}$  &  $\Sigma=\begin{bmatrix}
    -1 & 1 & 1 & 1 \\ 1 & -1 & 1 & 1 \\ 1 & 1 & -1 & 1 \\ 1 & 1 & 1 & -1
\end{bmatrix}, \quad M=\begin{bmatrix}
    -1 & 1 & 1 & 1 \\ 1 & -1 & 1 & 1 \\ 1 & 1 & -1 & 1 \\ 1 & 1 & 1 & -1
\end{bmatrix}, \quad P=(-1, 0), \quad Q=(-1, 0)$  &  \\ 
&    $y_2x_1 = x_1y_1+x_2y_1-x_1y_2+x_2y_2, \quad y_2x_2 =x_1y_1+x_2y_1+x_1y_2-x_2y_2$ & & & &  \\
\hline
\multirow{3}{*}{$\mathbb{T}$} &  $x_2x_1 = -x_1 x_2,\quad y_2y_1 = -y_1y_2$, & & &  & \\ 
& $y_1x_1 = -x_1y_1+x_2y_1+x_1y_2+x_2y_2,  \quad y_1x_2 = x_1y_1-x_2y_1+x_1y_2+x_2y_2$,  & $\Sigma_{11}=\begin{bmatrix}
    -1 & 1  \\ 1 & -1
\end{bmatrix}, \quad \Sigma_{12}=\begin{bmatrix}
    1 & 1  \\ 1 & 1 
\end{bmatrix}, \quad \Sigma_{21}=\begin{bmatrix}
    1 & 1  \\ 1 & 1 
\end{bmatrix}, \quad \Sigma_{22}=\begin{bmatrix}
    1 & -1  \\ -1 & 1 
\end{bmatrix}$ & $M_{11}=\begin{bmatrix}
    -1 & 1  \\ 1 & 1 
\end{bmatrix}, \quad M_{12}=\begin{bmatrix}
    1 & 1  \\ 1 & -1 
\end{bmatrix}, \quad M_{21}=\begin{bmatrix}
    1 & 1  \\ 1 & -1 
\end{bmatrix}, \quad M_{22}=\begin{bmatrix}
    -1 & 1  \\ 1 & 1 
\end{bmatrix}$  & $\Sigma=\begin{bmatrix}
    -1 & 1 & 1 & 1 \\ 1 & -1 & 1 & 1 \\ 1 & 1 & 1 & -1 \\ 1 & 1 & -1 & 1
\end{bmatrix}, \quad M=\begin{bmatrix}
    -1 & 1 & 1 & 1 \\ 1 & 1 & 1 & -1 \\ 1 & 1 & -1 & 1 \\ 1 & -1 & 1 & 1
\end{bmatrix}, \quad P=(-1, 0), \quad Q=(-1, 0)$  &  \\ 
&   $y_2x_1 = x_1y_1+x_2y_1+x_1y_2-x_2y_2, \quad y_2x_2 =x_1y_1+x_2y_1-x_1y_2+x_2y_2$ & &  & & \\
\hline
\end{tabular}
}
\end{center}
\end{table}
}}
\end{landscape}

\begin{landscape}
{\huge{
\begin{table}[h]
\caption{Double extensions}
\label{FourthTableDOE}
\begin{center}
\resizebox{20cm}{!}{
\setlength\extrarowheight{7pt}
\begin{tabular}{ |c|c|c|c|c|c| } 
\hline
Double extension & Relations defining the double extension  & $\Sigma_{ij}$ & $M_{ij}$ & Data $\{\Sigma, M, P, Q\}$ & Conditions
\\
\hline
\multirow{3}{*}{$\mathbb{U}$} &  $x_2x_1 = -x_1 x_2,\quad y_2y_1 = -y_1y_2$, &  & & & \\ 
&  $y_1x_1 = -x_1y_1+x_2y_1+x_1y_2+x_2y_2,  \quad y_1x_2 = x_1y_1+x_2y_1+x_1y_2-x_2y_2$, & $\Sigma_{11}=\begin{bmatrix}
    -1 & 1  \\ 1 & 1 
\end{bmatrix}, \quad \Sigma_{12}=\begin{bmatrix}
    1 & 1  \\ 1 & -1 
\end{bmatrix}, \quad \Sigma_{21}=\begin{bmatrix}
    1 & 1  \\ 1 & -1 
\end{bmatrix}, \quad \Sigma_{22}=\begin{bmatrix}
    -1 & 1  \\ 1 & 1 
\end{bmatrix}$ & $M_{11}=\begin{bmatrix}
    -1 & 1  \\ 1 & -1 
\end{bmatrix}, \quad M_{12}=\begin{bmatrix}
    1 & 1  \\ 1 & 1 
\end{bmatrix}, \quad M_{21}=\begin{bmatrix}
    1 & 1  \\ 1 & 1 
\end{bmatrix}, \quad M_{22}=\begin{bmatrix}
    1 & -1  \\ -1 & 1 
\end{bmatrix}$   & $\Sigma=\begin{bmatrix}
    -1 & 1 & 1 & 1 \\ 1 & 1 & 1 & -1 \\ 1 & 1 & -1 & 1 \\ 1 & -1 & 1 & 1
\end{bmatrix}, \quad M=\begin{bmatrix}
    -1 & 1 & 1 & 1 \\ 1 & -1 & 1 & 1 \\ 1 & 1 & 1 & -1 \\ 1 & 1 & -1 & 1
\end{bmatrix}, \quad P=(-1, 0), \quad Q=(-1, 0)$  &  \\ 
&    $y_2x_1 = x_1y_1+x_2y_1-x_1y_2+x_2y_2, \quad y_2x_2 =x_1y_1-x_2y_1+x_1y_2+x_2y_2$ & & & &  \\
\hline
\multirow{3}{*}{$\mathbb{V}$} & $x_2x_1 = x_1 x_2,\quad y_2y_1 = -y_1y_2$, & & &  &  \\ 
& $y_1x_1 = x_2y_1+x_1y_2,  \quad y_1x_2= x_2y_1$,  & $\Sigma_{11}=\begin{bmatrix}
    0 & 1  \\ 0 & 1 
\end{bmatrix}, \quad \Sigma_{12}=\begin{bmatrix}
    1 & 0  \\ 0 & 0 
\end{bmatrix}, \quad \Sigma_{21}=\begin{bmatrix}
    -1 & 1  \\ 0 & 0 
\end{bmatrix}, \quad \Sigma_{22}=\begin{bmatrix}
    0 & 0  \\ 0 & 1 
\end{bmatrix}$ & $M_{11}=\begin{bmatrix}
    0 & 1  \\ -1 & 0 
\end{bmatrix}, \quad M_{12}=\begin{bmatrix}
    1 & 0  \\ 1 & 0 
\end{bmatrix}, \quad M_{21}=\begin{bmatrix}
    0 & 0  \\ 0 & 0 
\end{bmatrix}, \quad M_{22}=\begin{bmatrix}
    1 & 0  \\ 0 & 1 
\end{bmatrix}$  & $\Sigma=\begin{bmatrix}
    0 & 1 & 1 & 0 \\ 0 & 1 & 0 & 0 \\ -1 & 1 & 0 & 0 \\ 0 & 0 & 0 & 1
\end{bmatrix}, \quad M=\begin{bmatrix}
    0 & 1 & 1 & 0 \\ -1 & 0 & 1 & 0 \\ 0 & 0 & 1 & 0 \\ 0 & 0 & 0 & 1
\end{bmatrix}, \quad P=(-1, 0), \quad Q=(1, 0)$  &  \\ 
&    $y_2x_1 = -x_1y_1+x_2y_1, \quad y_2x_2=x_2y_2$ & & & & \\
\hline
\multirow{3}{*}{$\mathbb{W}$} &  $x_2x_1 = x_1 x_2,\quad y_2y_1 = -y_1y_2$, & & &  & \\ 
&  $y_1x_1 = fx_2y_1+x_1y_2,  \quad y_1x_2= x_1y_1-x_2y_2$, & $\Sigma_{11}=\begin{bmatrix}
    0 & f  \\ 1 & 0 
\end{bmatrix}, \quad \Sigma_{12}=\begin{bmatrix}
    1 & 0  \\ 0 & -1 
\end{bmatrix}, \quad \Sigma_{21}=\begin{bmatrix}
    1 & 0  \\ 0 & -1 
\end{bmatrix}, \quad \Sigma_{22}=\begin{bmatrix}
    0 & f  \\ 1 & 0 
\end{bmatrix}$ & $M_{11}=\begin{bmatrix}
    0 & 1  \\ 1 & 0 
\end{bmatrix}, \quad M_{12}=\begin{bmatrix}
    f & 0  \\ 0 & f 
\end{bmatrix}, \quad M_{21}=\begin{bmatrix}
    1 & 0  \\ 0 & 1 
\end{bmatrix}, \quad M_{22}=\begin{bmatrix}
    0 & -1  \\ -1 & 0 
\end{bmatrix}$ &  $\Sigma=\begin{bmatrix}
    0 & f & 1 & 0 \\ 1 & 0 & 0 & -1 \\ 1 & 0 & 0 & f \\ 0 & -1 & 1 & 0
\end{bmatrix}, \quad M=\begin{bmatrix}
    0 & 1 & f & 0 \\ 1 & 0 & 0 & f \\ 1 & 0 & 0 & -1 \\ 0 & 1 & -1 & 0
\end{bmatrix}, \quad P=(-1, 0), \quad Q=(1, 0)$  & $f \not = -1$ \\ 
&    $y_2x_1 = x_1y_1+fx_2y_2, \quad y_2x_2=-x_2y_1+x_1y_2$ & & & & \\
\hline
\multirow{3}{*}{$\mathbb{X}$} &   $x_2x_1 = x_1 x_2,\quad y_2y_1 = -y_1y_2$, & & & & \\ 
&  $y_1x_1 = x_1y_2,  \quad y_1x_2= x_1y_2+x_2y_2$, & $\Sigma_{11}=\begin{bmatrix}
    0 & 0  \\ 0 & 0 
\end{bmatrix}, \quad \Sigma_{12}=\begin{bmatrix}
    1 & 0  \\ 1 & 1 
\end{bmatrix}, \quad \Sigma_{21}=\begin{bmatrix}
    1 & 0  \\ 1 & 1 
\end{bmatrix}, \quad \Sigma_{22}=\begin{bmatrix}
    0 & 0  \\ 0 & 0 
\end{bmatrix}$ & $M_{11}=\begin{bmatrix}
    0 & 1  \\ 1 & 0 
\end{bmatrix}, \quad M_{12}=\begin{bmatrix}
    0 & 0  \\ 0 & 0 
\end{bmatrix}, \quad M_{21}=\begin{bmatrix}
    1 & 1  \\ 1 & 0 
\end{bmatrix}, \quad M_{22}=\begin{bmatrix}
    0 & 1  \\ 1 & 0 
\end{bmatrix}$ & $\Sigma=\begin{bmatrix}
    0 & 0 & 1 & 0 \\ 0 & 0 & 1 & 1 \\ 1 & 0 & 0 & 0 \\ 1 & 1 & 0 & 0
\end{bmatrix}, \quad M=\begin{bmatrix}
    0 & 1 & 0 & 0 \\ 1 & 0 & 0 & 0 \\ 1 & 1 & 0 & 1 \\ 1 & 0 & 1 & 0
\end{bmatrix}, \quad P=(-1, 0), \quad Q=(1, 0)$  &  \\ 
&    $y_2x_1 = x_1y_1, \quad y_2x_2=x_1y_1+x_2y_1$ & & & & \\
\hline
\multirow{3}{*}{$\mathbb{Y}$} &   $x_2x_1 = x_1 x_2,\quad y_2y_1 = -y_1y_2$, & & & & \\ 
& $y_1x_1 = x_1y_1,  \quad y_1x_2= fx_1y_1-x_2y_1+x_1y_2$,  & $\Sigma_{11}=\begin{bmatrix}
    1 & 0  \\ f & -1 
\end{bmatrix}, \quad \Sigma_{12}=\begin{bmatrix}
    0 & 0  \\ 1 & 0 
\end{bmatrix}, \quad \Sigma_{21}=\begin{bmatrix}
    0 & 0  \\ 1 & 0 
\end{bmatrix}, \quad \Sigma_{22}=\begin{bmatrix}
    1 & 0  \\ f & -1 
\end{bmatrix}$ & $M_{11}=\begin{bmatrix}
    1 & 0  \\ 0 & 1 
\end{bmatrix}, \quad M_{12}=\begin{bmatrix}
    0 & 0  \\ 0 & 0 
\end{bmatrix}, \quad M_{21}=\begin{bmatrix}
    f & 1  \\ 1 & f 
\end{bmatrix}, \quad M_{22}=\begin{bmatrix}
    -1 & 0  \\ 0 & -1 
\end{bmatrix}$  & $\Sigma=\begin{bmatrix}
    1 & 0 & 0 & 0 \\ f & -1 & 1 & 0 \\ 0 & 0 & 1 & 0 \\ 1 & 0 & f & -1
\end{bmatrix}, \quad M=\begin{bmatrix}
    1 & 0 & 0 & 0 \\ 0 & 1 & 0 & 0 \\ f & 1 & -1 & 0 \\ 1 & f & 0 & -1
\end{bmatrix}, \quad P=(-1, 0), \quad Q=(1, 0)$  & $f$ is general \\ 
&    $ y_2x_1 = x_1y_2, \quad y_2x_2=x_1y_1+fx_1y_2-x_2y_2$ & & & & \\
\hline
\multirow{3}{*}{$\mathbb{Z}$} &  $x_2x_1 = -x_1 x_2,\quad y_2y_1 = y_1y_2$, & & &  & \\ 
& $y_1x_1 = x_1y_1+x_2y_2,  \quad y_1x_2 = x_2y_1+x_1y_2$,  & $\Sigma_{11}=\begin{bmatrix}
    1 & 0  \\ 0 & 1 
\end{bmatrix}, \quad \Sigma_{12}=\begin{bmatrix}
    0 & 1  \\ 1 & 0
\end{bmatrix}, \quad \Sigma_{21}=\begin{bmatrix}
    0 & f  \\ f & 0 
\end{bmatrix}, \quad \Sigma_{22}=\begin{bmatrix}
    -1 & 0  \\ 0 & -1 
\end{bmatrix}$ & $M_{11}=\begin{bmatrix}
    1 & 0  \\ 0 & -1 
\end{bmatrix}, \quad M_{12}=\begin{bmatrix}
    0 & 1  \\ f & 0 
\end{bmatrix}, \quad M_{21}=\begin{bmatrix}
    0 & 1  \\ f & 0 
\end{bmatrix}, \quad M_{22}=\begin{bmatrix}
    1 & 0  \\ 0 & -1 
\end{bmatrix}$  & $\Sigma=\begin{bmatrix}
    1 & 0 & 0 & 1 \\ 0 & 1 & 1 & 0 \\ 0 & f & -1 & 0 \\ f & 0 & 0 & -1
\end{bmatrix}, \quad M=\begin{bmatrix}
    1 & 0 & 0 & 1 \\ 0 & -1 & f & 0 \\ 0 & 1 & 1 & 0 \\ f & 0 & 0 & -1
\end{bmatrix}, \quad P=(1, 0), \quad Q=(-1, 0)$  & $f(1+f)\not = 0$ is general \\ 
&    $y_2x_1 = fx_2y_1-x_1y_2, \quad y_2x_2=fx_1y_1-x_2y_2$ & & & & \\
\hline
\end{tabular}
}
\end{center}
\end{table}
}}

\begin{remark} Zhang and Zhang \cite[Subcase 4.4.4]{ZhangZhang2009} formulated the relations of the algebra $\mathbb{Z}$. However, these contain one typo since the coefficients of the relations do not match the entries of the matrix $\Sigma$. Our version of these relations is presented in Table \ref{FourthTableDOE}. They wrote the relations $y_1x_1 = x_1y_2+x_2y_2$ and $y_2x_2=fx_1y_2-x_2y_2$, and the typo concerns that they considered the first factor $x_1y_2$ when it should be $x_1y_1$ according to matrix $\Sigma$.
\end{remark}
\end{landscape}

\section*{Acknowledgements}
The author would like to thank Arturo Niño\footnote{\texttt{diego.nino1@uexternado.edu.co}} and Andrés Riaño\footnote{\texttt{ajrianop@udistrital.edu.co}} for their valuable contributions to several computations appearing in this paper.

\end{document}